\documentclass[a4paper,draft]{amsart}
\usepackage{latexsym}\usepackage{ifthen}\usepackage[leqno]{amsmath}
\usepackage{enumerate}\usepackage{calc}\usepackage{hyphenat}
\usepackage{amstext,amsbsy,amsopn,amsthm,amsgen}
\usepackage{amsfonts,amscd,amsxtra,upref}
\usepackage{mathrsfs}\usepackage{euscript}\usepackage{amssymb}
\swapnumbers
\theoremstyle{plain}
\newtheorem{Thm}{Theorem}[section]
\newtheorem{Lem}[Thm]{Lemma}
\newtheorem{Cor}[Thm]{Corollary}
\newtheorem{Pro}[Thm]{Proposition}
\newtheorem{Prp}[Thm]{Properties}

\theoremstyle{definition}
\newtheorem{Def}[Thm]{Definition}
\newtheorem{Exm}[Thm]{Example}

\newtheorem{Prb}[Thm]{Problem}
\theoremstyle{remark}
\newtheorem{Rem}[Thm]{Remark}

\newcommand{\myEmail}{piotr.niemiec@uj.edu.pl}
\newcommand{\myAddress}[1]{\noindent{}\ITE{\equal{#1}{}}{}{Piotr Niemiec\\{}}
   In\-sty\-tut Ma\-te\-ma\-ty\-ki\\{}Wy\-dzia\l{} Ma\-te\-ma\-ty\-ki i In\-for\-ma\-ty\-ki\\{}
   U\-ni\-wer\-sy\-tet Ja\-giel\-lo\'{n}\-ski\\{}ul. \L{}o\-ja\-sie\-wi\-cza 6\\{}
   30-348 Kra\-k\'{o}w\\{}Poland}
\newcommand{\myData}[1][Piotr Niemiec]{\author[P. Niemiec]{Piotr Niemiec}\address{\myAddress{#1}}
   \email{\myEmail}}

\newcommand{\CCC}{\mathbb{C}}

\newcommand{\RRR}{\mathbb{R}}

\newcommand{\ZZZ}{\mathbb{Z}}
\newcommand{\BBb}{\CMcal{B}}

\newcommand{\HHh}{\CMcal{H}}

\newcommand{\Bb}{\mathfrak{B}}

\newcommand{\bbB}{\mathscr{B}}
\newcommand{\ddD}{\mathscr{D}}\newcommand{\eeE}{\mathscr{E}}

\newcommand{\mmM}{\mathscr{M}}

\newcommand{\ssS}{\mathscr{S}}

\newcommand{\zzZ}{\mathscr{Z}}

\newcommand{\SECT}[1]{\section{#1}\renewcommand{\theequation}{\arabic{section}-\arabic{equation}}
   \setcounter{equation}{0}}
\newcommand{\ITE}[3]{\ifthenelse{#1}{#2}{#3}}\newcommand{\ITEE}[3]{\ITE{\equal{#1}{#2}}{#3}{}}

\newcommand{\card}{\operatorname{card}}

\newcommand{\Diag}{\operatorname{Diag}}

\newcommand{\IM}{\operatorname{Im}}

\newcommand{\RE}{\operatorname{Re}}

\newcommand{\tr}{\operatorname{tr}}
\newcommand{\leqsl}{\leqslant}\newcommand{\geqsl}{\geqslant}
\newcommand{\epsi}{\varepsilon}\newcommand{\varempty}{\varnothing}\newcommand{\dd}{\colon}
\newcommand{\dint}[1]{\,\textup{d} #1}

\newcommand{\TFCAE}{The following conditions are equivalent:}
\newcommand{\tfcae}{the following conditions are equivalent:}

\newcommand{\COR}[1]{Corollary~\textup{\ref{cor:#1}}}
\newcommand{\DEF}[1]{Definition~\textup{\ref{def:#1}}}
\newcommand{\EXM}[1]{Example~\textup{\ref{exm:#1}}}
\newcommand{\LEM}[1]{Lemma~\textup{\ref{lem:#1}}}
\newcommand{\PRB}[1]{Problem~\textup{\ref{prb:#1}}}
\newcommand{\PRO}[1]{Proposition~\textup{\ref{pro:#1}}}

\newcommand{\THM}[1]{Theorem~\textup{\ref{thm:#1}}}

\newenvironment{cor}[1]{\begin{Cor}\label{cor:#1}}{\end{Cor}}
\newenvironment{dfn}[1]{\begin{Def}\label{def:#1}}{\end{Def}}
\newenvironment{exm}[1]{\begin{Exm}\label{exm:#1}}{\end{Exm}}
\newenvironment{lem}[1]{\begin{Lem}\label{lem:#1}}{\end{Lem}}
\newenvironment{prb}[1]{\begin{Prb}\label{prb:#1}}{\end{Prb}}
\newenvironment{pro}[1]{\begin{Pro}\label{pro:#1}}{\end{Pro}}

\newenvironment{rem}[1]{\begin{Rem}\label{rem:#1}}{\end{Rem}}
\newenvironment{thm}[1]{\begin{Thm}\label{thm:#1}}{\end{Thm}}

\newcommand{\op}{\textup{\textsf{op}}}\newcommand{\spc}{\operatorname{sp}} 
\newcommand{\DDB}{\textup{\textsf{DDB}}}\newcommand{\DDC}{\textup{\textsf{DDC}}} 
\newcommand{\TC}{\textup{\textsf{TC}}} 

\newcommand{\bibITEM}[2]{\ITE{\equal{#2}{}}{\bibitem{#1} }{\bibitem[#2]{#1} }}
\newcommand{\BIB}[8]{
   \bibITEM{#1}{#8} #2, \textit{#3}, #4{} \textbf{#5} (#6), #7.}
\newcommand{\myBIB}[7][P. Niemiec]{\ITE{\equal{#7}{*}\or\equal{#7}{**}}{}{#1, \textit{#2}, }
   #3{}\ITE{\equal{#4}{}}{}{ \textbf{#4}} (#5), #6\ITE{\equal{#7}{*}}{}{.}}
\newcommand{\BIb}[6]{
   \bibITEM{#1}{#6} #2, \textit{#3}, #4, #5.}

\newcommand{\jRN}[2][]{
   \ITEE{#2}{AdvM}{\ITE{\equal{#1}{+}}
      {Advances in Mathematics}{Adv. Math.}}
   \ITEE{#2}{AmJM}{\ITE{\equal{#1}{+}}
      {American Journal of Mathematics}{Amer. J. Math.}}
   \ITEE{#2}{BAMS}{\ITE{\equal{#1}{+}}
      {Bulletin of the American Mathematical Society}{Bull. Amer. Math. Soc.}}
   \ITEE{#2}{MZ}{\ITE{\equal{#1}{+}}
      {Math. Z.}{Math. Z.}}
   \ITEE{#2}{PJapAc}{\ITE{\equal{#1}{+}}
      {Proceedings of the Japan Academy}{Proc. Japan Acad.}}
   \ITEE{#2}{PNAS}{\ITE{\equal{#1}{+}}
      {Proceedings of the National Academy of Sciences of the United States of America}
      {Proc. Natl. Acad. Sci. USA}}
   \ITEE{#2}{SM}{\ITE{\equal{#1}{+}}
      {Studia Mathematica}{Studia Math.}}
   \ITEE{#2}{SurApprTh}{\ITE{\equal{#1}{+}}
      {Surveys in Approximation Theory}{Surv. Approx. Theory}}
   \ITEE{#2}{TAMS}{\ITE{\equal{#1}{+}}
      {Transactions of the American Mathematical Society}{Trans. Amer. Math. Soc.}}
   \ITEE{#2}{ZAngewM}{\ITE{\equal{#1}{+}}
      {Z. Angew. Math. Mech.}{Z. Angew. Math. Mech.}}
   }

\newcommand{\paplist}[3][]{
   \ITEE{#3}{ABAleksandrov,VVPeller,DSPotapov,FASukochev2011}{
      \BIB{#2}{A.B. Aleksandrov, V.V. Peller, D.S. Potapov and F.A. Sukochev}
         {Functions of normal operators under perturbations}
         {\jRN{AdvM}}{226}{2011}{5216--5251}{#1}}
   \ITEE{#3}{RBhatia1997}{
      \BIb{#2}{R. Bhatia}
         {Matrix Analysis}
         {Springer, New York}{1997}{#1}}
   \ITEE{#3}{FFBonsall,NJDuncan1973}{
      \BIb{#2}{F.F. Bonsall and N.J. Duncan}
         {Complete Normed Algebras}
         {Springer\hyp{}Verlag, Berlin}{1973}{#1}}
   \ITEE{#3}{CDeBoor2005}{
      \BIB{#2}{C. de Boor}
         {Divided differences}
         {\jRN{SurApprTh}}{1}{2005}{46--69}{#1}}
   \ITEE{#3}{JBConway1990}{
      \BIb{#2}{J.B. Conway}
         {A Course in Functional Analysis \textup{(Graduate Texts in Mathematics, vol. 96)}}
         {Springer, New York}{1990}{#1}}
   \ITEE{#3}{NDunford,JTSchwartz1971}{
      \BIb{#2}{N. Dunford and J.T. Schwartz}
         {Linear Operators, part III}
         {Wiley\hyp{}Interscience, New York}{1971}{#1}}
   \ITEE{#3}{BFuglede1950}{
      \BIB{#2}{B. Fuglede}
         {A commutativity theorem for normal operators}
         {\jRN{PNAS}}{36}{1950}{35--40}{#1}}
   \ITEE{#3}{TKato1973}{
      \BIB{#2}{T. Kato}
         {Continuity of the map $S \mapsto |S|$ for linear operators}
         {\jRN{PJapAc}}{49}{1973}{157--160}{#1}}
   \ITEE{#3}{KLowner1934}{
      \BIB{#2}{K. L\"{o}wner}
         {\"{U}ber monotone Matrixfunctionen}
         {\jRN{MZ}}{38}{1934}{177--216}{#1}}
   \ITEE{#3}{pn2002}{\bibITEM{#2}{#1} \mypaplist{pn1}{}}
   \ITEE{#3}{GOpitz1964}{
      \BIB{#2}{G. Opitz}
         {Steigungsmatrizen}
         {\jRN{ZAngewM}}{44}{1964}{T52--T54}{#1}}
   \ITEE{#3}{CRPutnam1951}{
      \BIB{#2}{C.R. Putnam}
         {On normal operators in Hilbert space}
         {\jRN{AmJM}}{73}{1951}{357--362}{#1}}
   \ITEE{#3}{DSerre2002}{
      \BIb{#2}{D. Serre}
         {Matrices: Theory and Applications \textup{(Graudate Texts in Mathematics 216)}}
         {Springer\hyp{}Verlag, New York}{2002}{#1}}
   \ITEE{#3}{HWhitney1934}{
      \BIB{#2}{H. Whitney}
         {Derivatives, difference quotients and Taylor's formula}
         {\jRN{BAMS}}{40}{1934}{89--94}{#1}}
   \ITEE{#3}{HWhitney1934a}{
      \BIB{#2}{H. Whitney}
         {Differentiable functions defined in closed sets. I}
         {\jRN{TAMS}}{36}{1934}{369--387}{#1}}
   }

\newcommand{\mypaplist}[3][]{
   \ITEE{#2}{pn1}{
      \myBIB{Separate and joint similarity to families of normal operators}
         {\jRN{SM}}{149}{2002}{39--62}{#3}}
   }

\begin{document}

\title{Functional calculus for diagonalizable matrices}
\myData\thanks{The author gratefully acknowledges the assistance of the Polish Ministry 
   of Sciences and Higher Education grant NN201~546438 for the years 2010--2013.}
\begin{abstract}
For an arbitrary function $f\dd \Omega \to \CCC$ (where $\Omega \subset \CCC$) and a positive
integer $k$ let $f_{\op}\dd \ddD_k(\Omega) \ni X \mapsto f[X] \in \ddD_k(\CCC)$ where
$\ddD_k(\Omega)$ consists of all $k \times k$ matrices similar to diagonal whose all eigenvalues lie
in $\Omega$ be the function defined as follows: $f[P \Diag(\lambda_1,\ldots,\lambda_k) P^{-1}] =
P \Diag(f(\lambda_1),\ldots,f(\lambda_k)) P^{-1}$ for arbitrary $\lambda_1,\ldots,\lambda_k \in
\Omega$ and an invertible $k \times k$ matrix $P$. The aim of the paper is to fully answer
the question of when $f_{\op}$ is continuous for fixed $k$. In particular, it is shown that
if $\Omega$ is open in $\CCC$, then $f_{\op}$ is continuous for fixed $k \geqsl 3$ iff $f$ is
holomorphic; and if $\Omega$ is an interval in $\RRR$ and $k \geqsl 3$, then $f_{\op}$ is continuous
on $\ddD_k(\Omega)$ iff $f \in C^{k-2}(\Omega)$ and $f^{(k-2)}$ is locally Lipschitz in $\Omega$.
Also a full characterization is given when the domain of $f$ is arbitrary as well as when $f_{\op}$
acts on infinite-dimensional (diagonalizable) matrices.
\end{abstract}
\subjclass[2010]{Primary 47A60, 26E10; Secondary 47A56.}
\keywords{Diagonalizable matrix; functional calculus; divided difference; smooth function; 
   scalar operator.}
\maketitle


\SECT{Introduction}

Continuous (or Borel) functional calculus for normal operators is an interesting concept widely
investigated in operator theory. There are many spectacular results dealing with this concept, e.g.
Loewner's theorem \cite{low} on operator monotone functions (for other proof and a discussion see
Chapter~V in \cite{bha}); Aleksandrov's-Peller's-Potapov's-Sukochev's theorem \cite{ap+} on operator
H\"{o}lder functions which turn out to coincide with H\"{o}lder functions---this is in contrast
to Lipschitz functions which may not be operator Lipschitz (consult e.g. \cite{kat}). It is quite
natural to extend the above functional calculus to other operators (or matrices) than normal.
The simplest class applicable here is formed by diagonalizable matrices (that is, matrices similar
to diagonal). In this way for any function $f\dd \Omega \to \CCC$ and each positive integer $k$
we may properly define  a matrix-valued function $f_{\op}\dd \ddD_k(\Omega) \to \ddD_k(\CCC)$ where
$f_{\op}$, $\ddD_k(\Omega)$ and $\ddD_k(\CCC)$ are as in Abstract. Having such an extended
functional calculus, we may pose analogous questions as in case of normal matrices. Most basic among
them is the continuity of $f_{\op}$. More precisely, we may study the following issue:
\begin{prb}{!}
Given a set $\Omega \subset \CCC$ and a positive integer $k$, characterize all functions $f\dd\
\Omega \to \CCC$ for which the function $f_{\op}\dd \ddD_k(\Omega) \to \ddD_k(\CCC)$ is continuous.
\end{prb}

At first sight, one may suspect that the characterization is `trivial', that is, that the continuity
of $f$ is sufficient for the continuity of $f_{\op}$. Surprisingly, it turns out that this
supposition fails even for $k = 2$. The main aim of the paper is to give a full answer to \PRB{!}.
Our general characterization involves so-called divided differences (see Section~3). However,
in most of practical cases, i.e. when the set $\Omega$ is open in $\CCC$ or a subinterval of $\RRR$,
the criterion may simply be formulated. To this end, denote by $\mmM_k(\Omega)$ the set of all
$k \times k$ matrices whose all eigenvalues lie in $\Omega$. Additionally, put
\begin{equation}\label{eqn:singular}
\zzZ_k = \{X \in \mmM_k(\CCC)\dd\ (X - \lambda I_k)^{k-1} \neq 0 = (X - \lambda I_k)^k
\textup{ for some } \lambda \in \CCC\}
\end{equation}
(where $I_k$ stands for the $k \times k$ unit matrix). Two of our main results read as follows:

\begin{thm}{main}
Let $k \geqsl 3$ be fixed.
\begin{enumerate}[\upshape(A)]
\item Let $\Omega$ be an open set in $\CCC$ and $f\dd \Omega \to \CCC$ be any function. \TFCAE
   \begin{enumerate}[\upshape(i)]
   \item $f_{\op}\dd \ddD_k(\Omega) \to \ddD_k(\CCC)$ is continuous;
   \item $f_{\op}\dd \ddD_2(\Omega) \to \ddD_2(\CCC)$ extends to a continuous function
      of $\mmM_2(\Omega)$ into $\mmM_2(\CCC)$;
   \item $f$ is holomorphic.
   \end{enumerate}
\item Let $\Omega$ be a subinterval of $\RRR$ or an open set in $\RRR$, and let $f\dd \Omega \to
   \CCC$ be arbitrary. \TFCAE
   \begin{enumerate}[\upshape(i)]
   \item $f_{\op}\dd \ddD_k(\Omega) \to \ddD_k(\CCC)$ is continuous;
   \item $f_{\op}\dd \ddD_k(\Omega) \to \ddD_k(\CCC)$ extends to a continuous function
      of $\mmM_k(\Omega) \setminus \zzZ_k$ into $\mmM_k(\CCC)$;
   \item $f$ is of class $C^{k-2}$ and $f^{(k-2)}$ is locally Lipschitz.
   \end{enumerate}
   Moreover, $f_{\op}$ extends to a continuous function of $\mmM_k(\Omega)$ into $\mmM_k(\CCC)$ iff
   $f$ is of class $C^{k-1}$. In particular, $f_{\op}\dd \ddD_n(\Omega) \to \ddD_n(\CCC)$
   is continuous for each $n \geqsl 1$ iff $f$ is of class $C^{\infty}$, and then $f_{\op}$ extends
   to a continuous function of $\mmM_n(\Omega)$ to $\mmM_n(\CCC)$ for all $n$.
\end{enumerate}
\end{thm}

\begin{thm}{holo}
Let $\Omega \subset \CCC$ and $f\dd \Omega \to \CCC$ be arbitrary. \TFCAE
\begin{enumerate}[\upshape(i)]
\item for any $\lambda \in \Omega$ there are positive real constants $M = M(\lambda)$ and $\epsi =
   \epsi(\lambda)$ such that $\|f[X]\| \leqsl M$ whenever $\|X - \lambda I_n\| \leqsl \epsi$,
   $X \in \ddD_n(\Omega)$ and each $n$;
\item for any $\lambda \in \Omega$ and each $\epsi > 0$ there exists $\delta > 0$ such that $\|f[X]
   - f[\lambda I_n]\| \leqsl \epsi$ whenever $\|X - \lambda I_n\| \leqsl \delta$, $X \in
   \ddD_n(\Omega)$ and $n$ is arbitrary;
\item $f$ extends to a holomorphic function of an open superset of $\Omega$ into $\CCC$.
\end{enumerate}
\end{thm}

\THM{holo} shall be applied to characterize those functions $f$ for which $f_{\op}$, as a function
acting on (infinite-dimensional) diagonalizable (or so-called scalar) operators, is continuous (see
\PRO{holo} below).\par
The concept of operator H\"{o}lder or operator Lipschitz functions may simply be adapted
to the context of functional calculus for diagonalizable matrices (while, in the opposite, operator
monotonicity makes no longer sense in this realm). Much weaker property in this direction is uniform
continuity. It turns out that for diagonalizable matrices the property of `matrix uniform
continuity' becomes trivial, as shown by

\begin{pro}{uniform}
Let $\Omega \subset \CCC$, $f\dd \Omega \to \CCC$ and $k \geqsl 3$ be arbitrary \textup{(}or $k = 2$
and $\Omega$ has a cluster point in $\CCC$\textup{)}. \TFCAE
\begin{enumerate}[\upshape(i)]
\item there are positive real constants $M$ and $\epsi$ such that $\|f[X] - f[Y]\| \leqsl M$
   whenever $X, Y \in \ddD_k(\Omega)$ are such that $\|X - Y\| \leqsl \epsi$;
\item $f_{\op}\dd \ddD_k(\Omega) \to \ddD_k(\CCC)$ is uniformly continuous;
\item $f$ is of the form $f(z) = az + b$ for some $a, b \in \CCC$.
\end{enumerate}
\end{pro}

\THM{holo} and the results on the functional calculus for bounded diagonalizable Hilbert space
operators presented in Section~6 assert that the classical holomorphic functional calculus for
single Banach algebra elements (see e.g. \cite[I.\S7]{b-d}) is as rich as possible even for
diagonalizable Hilbert space operators when we require its continuity.\par
The paper is organized as follows. In Section~2 we present the concept of T-differentiability
(in the complex sense) of functions defined on arbitrary subsets of the complex plane. The idea
is based on the Taylor expansion of holomorphic functions. For functions sufficiently many times
T-differentiable we define their (in a sense artificial) calculus for matrices, whose `naturalness'
will later be justified by its continuity in case of more `regular' functions. In the next section
we deal with the divided differences for arbitrary functions. We give there certain criteria for
the differentiability of a function by means of its divided differences. These results will find
applications in the next part, where we give a full characterization of those functions $f\dd \Omega
\to \CCC$ for which $f_{\op}\dd \ddD_k(\Omega) \to \ddD_k(\CCC)$ is continuous. This section
contains also the proof of \THM{main}. Section~5 is devoted to the aspects of uniform continuity
of $f_{\op}$ and contains the proofs of \THM{holo} and \PRO{uniform}, while the last part
(Section~6) discusses infinite-dimensional case.

\subsection*{Notation and terminology} Whenever $\lambda_1,\ldots,\lambda_n$ are arbitrary complex
numbers, $\Diag(\lambda_1,\ldots,\lambda_n)$ denotes the diagonal $n \times n$ matrix whose diagonal
entries are precisely $\lambda_1,\ldots,\lambda_n$. For a complex $n \times n$ matrix $A$, $\|A\|$
stands for the \textit{operator norm} of $A = [a_{jk}]$ induced by the standard inner product
on $\CCC^n$, that is:
$$\|A\| = \max\Bigl\{\bigl|\sum_{j,k=1}^n a_{jk} z_k \bar{w}_j\bigr|\dd\ z_1,w_1,\ldots,z_n,w_n \in
\CCC,\ \sum_{k=1}^n |z_k|^2 = \sum_{j=1}^n |w_j|^2 = 1\Bigr\}.$$
The matrix $A$ is \textit{diagonalizable} iff it is similar to a diagonal one, i.e. if $P A P^{-1}$
is diagonal for some invertible $n \times n$ (complex) matrix $P$. The set of all complex
eigenvalues of $A$ (the \textit{spectrum} of $A$) is denoted by $\spc(A)$. We denote by $\mu_A$
the \textit{minimal} polynomial for $A$; that is, $\mu_A$ is a monic polynomial of a minimal degree
such that $\mu_A[A] = 0$.\par
Everywhere in this paper $\Omega$ denotes a totally arbitrary nonempty subset of the complex plane
and $f$ is any function of $\Omega$ into $\CCC$; $\bar{\Omega}$ and $\Omega' (\subset \Omega)$
denote, respectively, the closure of $\Omega$ in $\CCC$ and the set of all cluster points
of $\Omega$, that is, $\Omega'$ consists of all $z \in \Omega$ such that $z$ belongs to the closure
of $\Omega \setminus \{z\}$. For simplicity, the notation $w\stackrel{\Omega}{\to}z$ will mean that
$w$ runs over $\Omega \setminus \{z\}$ and tends to $z$. A complex-valued function defined
on a subset of $\CCC^n$ is said to be \textit{locally Lipschitz} iff every point of its domain has
a relative neighbourhood (i.e. relatively open in the domain of the function) on which the function
is Lipschitz. Subintervals of the real line are assumed to be nondegenerate.\par
For any $k \geqsl 1$, let us denote by $\mmM_k(\Omega)$ and $\ddD_k(\Omega)$, respectively, the sets
of all complex $k \times k$ matrices $X$ with $\spc(X) \subset \Omega$ and all such diagonalizable
matrices. Observe that:
\begin{enumerate}[(M1)]
\item $\ddD_k(\Omega) = \{X \in \mmM_k(\Omega)\dd\ \textup{all roots of $\mu_X$ are simple}\}$.
\end{enumerate}
Further, we put:
\begin{enumerate}[(M1)]\addtocounter{enumi}{1}
\item $\mmM_k^o(\Omega) = \{X \in \mmM_k(\Omega)\dd\ \textup{all possible multiple roots of $\mu_X$
   belong to $\Omega'$}\}$.
\end{enumerate}

\SECT{Abstract concept of differentiability}

The idea of the Taylor expansion enables us to introduce the following

\begin{dfn}{diff}
Let $k \geqsl 1$. A function $f\dd \Omega \to \CCC$ is said to be \textit{$k$ times T-differentiable
at a point $a \in \Omega'$} (the prefix `T-' is to emphasize the role of the Taylor expansion) iff
there are complex numbers $u_0,\ldots,u_k$ and a function $\tau_a\dd \Omega \to \CCC$ such that
\begin{equation}\label{eqn:Taylor}
\begin{array}{l}
\bullet\ f(z) = \sum_{j=0}^k \frac{u_j}{j!} (z - a)^j + \tau_a(z) (z - a)^k \qquad (z \in \Omega),\\
\bullet\ \tau_a(a) = 0,\\
\bullet\ \tau_a \textup{ is continuous at } a.
\end{array}
\end{equation}
It is easy to observe that the numbers $u_0,\ldots,u_k$ and the function $\tau_a$ are uniquely
determined by \eqref{eqn:Taylor} and thus we may put $f^{(j)}(a) := u_j\ (j=0,\ldots,k)$ whenever
\eqref{eqn:Taylor} holds. (In particular, $f^{(0)}(a) = f(a)$.) We say $f$ is \textit{$k$ times
T-differentiable} iff it is so at each point of $\Omega'$. In that case we call the function
$f^{(j)}\dd \Omega' \to \CCC$ the \textit{$j$th T-derivative} of $f$. Finally, $f$ is said to be
\textit{of class $\TC^k$} if $f$ is $k$ times T-differentiable and the functions
$f',\ldots,f^{(k)}\dd \Omega' \to \CCC$ as well as
$$\tau\dd \Omega' \times \Omega \ni (a,x) \mapsto \tau_a(x) \in \CCC$$
(where $\tau_a$ is as in \eqref{eqn:Taylor}) are continuous. When $f$ is of class $\TC^k$ for each
$k$, we express this by writing that $f$ is \textit{of class $\TC^{\infty}$}. Additionally, we call
the function $f$ \textit{of class $\TC^0$} if it is continuous, and we identify $f^{(0)}$ with $f$
(so, the domain of $f^{(0)}$ coincides with $\Omega$, which may differ from $\Omega'$).
\end{dfn}

The reader should notice that if $f$ is T-differentiable, then $f$ is continuous and $f'(z) =
\lim_{w\stackrel{\Omega}{\to}z} \frac{f(w) - f(z)}{w - z}$ for each $z \in \Omega'$. In particular,
if $\Omega$ is open in $\CCC$, then $f$ is $T$-differentiable iff it is of class $\TC^{\infty}$, iff
it is holomorhic. Much more difficult in proving is the following result, due to Whitney \cite{wh1}.

\begin{pro}{whitney1}
If $\Omega \subset \RRR$ is a subinterval or an open subset of the real line, then for arbitrary $k
\in \{1,2,\ldots\}$ a function $f\dd \Omega \to \CCC$ is of class $\TC^k$ iff it is of class $C^k$.
\end{pro}
\begin{proof}
Sufficiency follows from Taylor's theorem, while necessity is a consequence
of \cite[Theorem~3]{wh1} (indeed, due to that result, a function is of class $C^k$ provided
$\tau_a(x)$ tends to $0$ when $x$ tends to $a$ and this convergence is uniform on compact sets).
\end{proof}

We have introduced T-differentiable functions in order to extend the classical (polynomial
or holomorphic) functional calculus for matrices as widely as possible. First of all note that
if $X$ is a square matrix with $\mu_X(z) = \prod_{j=1}^m (z - \lambda_j)^{p_j}$ (where $\lambda_j$'s
are different) and $P$ and $Q$ are two arbitrary (complex) polynomials such that $P^{(s)}(\lambda_j)
= Q^{(s)}(\lambda_j)$ for any $j \in \{1,\ldots,m\}$ and $s \in \{0,\ldots,p_j-1\}$, then $P[X] =
Q[X]$. This simple observation leads us to the following

\begin{dfn}{T-calc}
Let $f\dd \Omega \to \CCC$ be of class $\TC^k$ (where $k \in \{0,1,2,\ldots,\infty\}$). Let $X \in
\mmM_n^o(\Omega)$ (cf. (M2)) be a matrix such that $\mu_X(z) = \prod_{j=1}^m (z - \lambda_j)^{p_j}$
(where $\lambda_j$'s are different) and $p_j \leqsl k+1$ for each $j$. Let $P$ be a polynomial such
that for each $j \in \{1,\ldots,m\}$:
\begin{itemize}
\item $P^{(s)}(\lambda_j) = f^{(s)}(\lambda_j)$ for $s \in \{0,\ldots,p_j-1\}$ provided $\lambda_j
   \in \Omega'$;
\item $P(\lambda_j) = f(\lambda_j)$ provided $\lambda_j \notin \Omega'$.
\end{itemize}
We define the matrix $f[X]$ as $P[X]$. The note preceding the definition shows that $f[X]$ is well
defined---that is, it is independent of the choice of $P$ (recall that $p_j = 1$ if $\lambda_j
\notin \Omega'$, since $X \in \mmM_n^o(\Omega)$). In particular, in this way we obtain functions
\begin{equation}\label{eqn:T-op}
\begin{array}{l}
f_{\op}\dd \mmM_n^o(\Omega) \ni X \mapsto f[X] \in \mmM_n(\CCC) \qquad (0 < n < k+2),\\
f_{\op}\dd \mmM_n^o(\Omega) \setminus \zzZ_n \ni X \mapsto f[X] \in \mmM_n(\CCC) \qquad (n = k+2)
\end{array}
\end{equation}
(where $\zzZ_n$ is as in \eqref{eqn:singular}). We call the assignment $f \mapsto f_{\op}$
the \textit{extended functional calculus} for matrices.
\end{dfn}

Although the extended functional calculus strikely resembles holomorphic, there is no algebraic nor
`practical' justification (apart from the calculus for diagonalizable matrices) of the approach
introduced above. Since $\ddD_n(\Omega)$ is dense in $\mmM_n^o(\Omega)$ (to convince of that, use
e.g. Jordan's matrix decomposition theorem) and $f_{\op}$ is naturally defined on $\ddD_n(\Omega)$
(for totally arbitrary functions $f\dd \Omega \to \CCC$), thus the continuity of $f_{\op}$
(on $\mmM_n^o(\Omega)$) would be a sufficiently good justification. In the course of our research
on this issue, it turned out that if only $f_{\op}$ is continuous on $\ddD_n(\Omega)$ (where
$n > 1$), then $f$ is of class $\TC^{n-2}$ and $f_{\op}$ (defined as above) is automatically
continuous on $\mmM_n^o(\Omega) \setminus \zzZ_n$ as well as on $\mmM_{n-1}^o(\Omega)$. These two
results motivated us to introduce \DEF{T-calc}. However, it is worth noting that being of class
$\TC^{\infty}$ for a function $f$ is insufficient for $f_{\op}$ to be continuous on $\ddD_2(\Omega)$
in general (see \EXM{TC-dis} below). To hit the mark of the problem, a stronger notion
of differentiability is needed, which we now turn to.

\SECT{Divided differences}

We begin this part with recalling the concept of \textit{divided differences} and their basic
properties. For each $n \geqsl 1$ let $\Omega^{(n)}$ be the set of all vectors $(z_1,\ldots,z_n) \in
\Omega^n$ whose all coordinates are different. Further, let $\Omega^{[n]}$ stand for the set of all
vectors $(z_1,\ldots,z_n) \in \Omega^n$ satisfying the following condition: if $z_j = z_k$ for some
distinct $j$ and $k$, then $z_j \in \Omega'$. Notice that the closure of $\Omega^{(n)}$
(in $\CCC^n$) coincides with $\bar{\Omega}^{[n]}$. The \textit{divided difference} of a function
$f\dd \Omega \to \CCC$ at $(z_1,\ldots,z_n) \in \Omega^{(n)}$, denoted
by $\Delta(z_1,\ldots,z_n) f$, is defined by induction on $n \geqsl 1$ as follows:
\begin{itemize}
\item $\Delta(z) f = f(z)$ for any $z \in \Omega^{(1)} (= \Omega)$;
\item $\Delta(z_1,\ldots,z_n) f = [\Delta(z_2,\ldots,z_n)f - \Delta(z_1,\ldots,z_{n-1}) f] /
   (z_n - z_1)$ for $n > 1$.
\end{itemize}
Let us now list two most important for us properties of the divided differences:
\begin{enumerate}[(DD1)]
\item The divided differences are symmetric; that is, whenever $(z_1,\ldots,z_n) \in \Omega^{(n)}$
   and $\sigma$ is a permutation of $\{1,\ldots,n\}$, then
   $\Delta(z_{\sigma(1)},\ldots,z_{\sigma(n)}) f = \Delta(z_1,\ldots,z_n) f$.
\item For any $(z_1,\ldots,z_n) \in \Omega^{(n)}$ (where $n > 1$), the polynomial $W(z) := f(z_1) +
   \sum_{k=2}^n \Delta(z_1,\ldots,z_k) f \cdot \prod_{j=1}^{k-1} (z - z_j)$ satisfies the equations
   $W(z_j) = f(z_j)$ for $j = 1,\ldots,n$.
\end{enumerate}
For the proofs of the above facts and a more detailed discussion on divided differences, consult
\cite{cdb}.\par
Our interest are functions whose divided differences satisfy some additional conditions. To this
end, we introduce

\begin{dfn}{DD}
Let $k \in \{0,1,2,\ldots\}$. A function $f\dd \Omega \to \CCC$ is said to be \textit{of class
$\DDB^k$}, in symbols $f \in \DDB^k(\Omega)$ [`DD' and `B' are the first letters of `divided
differences' and `bounded'], if for any $z \in \Omega'$ there are positive real constants $M =
M(\lambda)$ and $\epsi = \epsi(\lambda)$ such that $|\Delta(z_1,\ldots,z_{k+1}) f| \leqsl M$
whenever $(z_1,\ldots,z_{k+1}) \in \Omega^{(k+1)}$ and $|z_j - z| \leqsl \epsi$. The function $f$ is
\textit{of class $\DDB^{\infty}$} ($f \in \DDB^{\infty}(\Omega)$) if it is of class $\DDB^k$ for any
$k$.\par
Similarly, $f$ is said to be \textit{of class $\DDC^k$}, in symbols $f \in \DDC^k(\Omega)$ [`C' is
the first letter of `continuous'], if the function $\Omega^{(k+1)} \ni (z_1,\ldots,z_{k+1}) \mapsto
\Delta(z_1,\ldots,z_{k+1}) f \in \CCC$ has (finite) limit at $(z,\ldots,z) \in \CCC^{k+1}$ for each
$z \in \Omega'$. Finally, $f$ is \textit{of class $\DDC^{\infty}$} ($f \in \DDC^{\infty}(\Omega)$)
if $f \in \DDC^k(\Omega)$ for each $k$.
\end{dfn}

It follows from the very definitions that $\DDC^k(\Omega) \subset \DDB^k(\Omega)$ for any $k$.
At first sight, it may seem that the classes $\DDB$ and $\DDC$ have not much more in common (for
example, $f \in \DDB^0(\Omega)$ iff $f$ is locally bounded, while $f \in \DDC^0(\Omega)$ iff $f$ is
continuous). Therefore the following result may be surprising.

\begin{pro}{B-C}
For any $n \in \{1,2,3,\ldots\}$, $\DDB^n(\Omega) \subset \DDC^{n-1}(\Omega)$. Moreover, if $f \in
\DDB^n(\Omega)$, then there exists a locally compact set $G$, $\Omega \subset G \subset
\bar{\Omega}$, such that for any $k \in \{1,\ldots,n\}$ the function $\Omega^{(k)} \ni
(z_1,\ldots,z_k) \mapsto \Delta(z_1,\ldots,z_k)f \in \CCC$ extends to a locally Lipschitz symmetric
function of $G^{[k]}$ into $\CCC$.\par
In particular, $\DDB^{\infty}(\Omega) = \DDC^{\infty}(\Omega)$.
\end{pro}
\begin{proof}
Clearly, it is enough to prove the claims of the first paragraph of the proposition. We shall do
this by induction on $n$. To simplify the argument, for each $n > 0$ put
\begin{equation}\label{eqn:diag}
\Theta_n = \{(z,\ldots,z)\dd\ z \in \CCC\} \subset \CCC^n.
\end{equation}
When $n = 1$, the assumption that $f \in \DDB^1(\Omega)$ means that $f$ is locally Lipschitz at each
point of $\Omega'$. Since the other points of $\Omega$ are isolated, we infer that $f$ is locally
Lipschitz. So, for each $z \in \Omega$ there are positive real constants $\epsi_z$ and $M_z$ such
that $|f(w) - f(w')| \leqsl M_z |w - w'|$ for any $w, w' \in \Omega \cap B(z,\epsi_z)$ where
\begin{equation}\label{eqn:ball}
B(z,\epsi_z) = \{w \in \CCC\dd\ |w - z| < \epsi_z\}.
\end{equation}
We infer from the completeness of $\CCC$ that the restriction of $f$ to $\Omega \cap B(z,\epsi_z)$
extends to a Lipschitz function $g_z\dd \bar{\Omega} \cap B(z,\epsi_z) \to \CCC$. It is readily seen
that the functions $g_z$'s ($z \in \Omega$) agree and hence their union is locally Lipschitz at each
point of its domain $G := \bar{\Omega} \cap \bigcup_{z\in\Omega} B(z,\epsi_z)$. It remains to note
that $G$ is locally compact as the intersection of an open and a closed set.\par
Now assume the assertion holds for $n - 1$ (where $n > 1$); take $f \in \DDB^n(\Omega)$ and fix $z
\in \Omega'$. Let $\epsi_z > 0$ and $M_z > 0$ be such that $|\Delta(z_1,\ldots,z_{n+1})f| \leqsl
M_z$ whenever $z_1,\ldots,z_{n+1} \in \Omega \cap B(z,\epsi_z)$ are different. This means that
\begin{equation}\label{eqn:aux1}
|\Delta(z_1,\ldots,z_n)f - \Delta(z_2,\ldots,z_{n+1})f| \leqsl M_z |z_1 - z_{n+1}|
\end{equation}
for $(z_1,\ldots,z_{n+1}) \in \Omega^{(n+1)} \cap [B(z,\epsi_z)]^{n+1}$. Let us now show that
\begin{equation}\label{eqn:loc-L}
|\Delta(z_1,\ldots,z_n)f - \Delta(w_1,\ldots,w_n)f| \leqsl M_z \sum_{j=1}^n |z_j - w_j|
\end{equation}
for any $(z_1,\ldots,z_n),(w_1,\ldots,w_n) \in \Omega^{(n)} \cap [B(z,\epsi_z)]^n$. To this end,
put $I := \{z_1,\ldots,z_n\} \cap \{w_1,\ldots,w_n\}$. Involving (DD1) and permuting both
the systems $(z_1,\ldots,z_n)$ and $(w_1,\ldots,w_n)$ by means of a common permutation, we may and
do assume that $I = \{z_j\dd\ j \leqsl k\}$ for some $k \leqsl n$ ($k$ may be equal to $0$). If $k =
n$, we are done thanks to (DD1). Hence, we may and do assume that $k < n$. For $s \in
\{1,\ldots,k\}$ denote by $\sigma(s)$ a unique index $j$ for which $w_j = z_s$. Now fix for a moment
$s \in \{k+1,\ldots,n\}$ and put $\nu_s(1) = s$. Further we make use of induction: assume $\nu_s(j)$
is already defined for some $j \geqsl 1$. If $w_{\nu_s(j)} \in I$, define $\nu_s(j+1)$ as a unique
index $j'$ for which $w_{\nu_s(j)} = z_{j'}$. Otherwise put $m(s) = j$, $\sigma(s) = \nu_s(j)$ and
finish the construction for $s$. In this way we obtain a sequence $z_s,w_{\nu_s(1)},\ldots,
w_{\nu_s(m(s))}$ such that $w_{\nu_s(m(s))} = w_{\sigma(s)} \notin \{z_1,\ldots,z_n\}$ and
\begin{itemize}
\item[($\bullet$)] $w_{\nu_s(j)} = z_{\nu_s(j+1)} \in I$ for $0 \leqsl j < m(s)$ with convention
   that $w_{\nu_s(0)} := z_s$.
\end{itemize}
The above construction shows also that
\begin{itemize}
\item[($\bullet\bullet$)] $\nu_s(j) \neq \nu_{s'}(j')$ provided $s, s' > k$, $0 < j \leqsl m(s)$,
   $0 < j' \leqsl m(s')$ and $(s,j) \neq (s',j')$;
\end{itemize}
and $\sigma\dd \{1,\ldots,n\} \to \{1,\ldots,n\}$ is a permutation. Now (DD1) yields that
\begin{multline*}
|\Delta(z_1,\ldots,z_n)f - \Delta(w_1,\ldots,w_n)f| = |\Delta(z_1,\ldots,z_n)f
- \Delta(w_{\sigma(1)},\ldots,w_{\sigma(n)})f| =\\= |\Delta(z_1,\ldots,z_k,z_{k+1},\ldots,z_n)f
- \Delta(z_1,\ldots,z_k,w_{\sigma(k+1)},\ldots,w_{\sigma(n)})f| \leqsl\\\leqsl
\sum_{s=k+1}^n |\Delta(z_1,\ldots,z_s,w_{\sigma(s+1)},\ldots,w_{\sigma(n)})f
- \Delta(z_1,\ldots,z_{s-1},w_{\sigma(s)},\ldots,w_{\sigma(n)})f|.
\end{multline*}
Thanks to \eqref{eqn:aux1} (and again (DD1)), we may continue the above estimations as follows
(see the convention in ($\bullet$)):
\begin{multline*}
|\Delta(z_1,\ldots,z_n)f - \Delta(w_1,\ldots,w_n)f| \leqsl M_z \sum_{s=k+1}^n |z_s - w_{\sigma(s)}|
=\\= M_z \sum_{s=k+1}^n |z_s - w_{\nu_s(m_s)}| \leqsl M_z \sum_{s=k+1}^n \sum_{j=1}^{m(s)}
|w_{\nu_s(j-1)} - w_{\nu_s(j)}| \leqsl\\\leqsl M_z \sum_{s=k+1}^n \sum_{j=1}^{m(s)} |z_{\nu_s(j)}
- w_{\nu_s(j)}| \leqsl M_z \sum_{j=1}^n |z_j - w_j|
\end{multline*}
(cf. ($\bullet\bullet$)), which yields \eqref{eqn:aux1}.\par
The inequality \eqref{eqn:loc-L} implies that there is a Lipschitz function
$$g_z\dd \bar{\Omega}^{[n]} \cap [B(z,\epsi_z)]^n \to \CCC$$ such that
\begin{equation}\label{eqn:aux2}
g_z(z_1,\ldots,z_n) = \Delta(z_1,\ldots,z_n) f
\end{equation}
for $(z_1,\ldots,z_n) \in \Omega^{(n)} \cap [B(z,\epsi_z)]^n$. The density argument combined with
(DD1) yields the symmetry of $g_z$.\par
The above argument shows that $f \in \DDC^{n-1}(\Omega) (\subset \DDB^{n-1}(\Omega))$. So,
it follows from the induction hypothesis that there are a locally compact set $G_0$ ($\Omega \subset
G_0 \subset \bar{\Omega}$) and locally Lipschitz symmetric functions $F_k\dd G_0^{[k]} \to \CCC$
($k=1,\ldots,n-1$) such that
\begin{equation}\label{eqn:Fk}
F_k(z_1,\ldots,z_k) = \Delta(z_1,\ldots,z_k) f
\end{equation}
whenever $(z_1,\ldots,z_k) \in \Omega^{(k)}$. For simplicity, we introduce the following notation:
for any $z = (z_1,\ldots,z_n) \in \CCC^n$ and $j \in \{1,\ldots,n\}$, let $z'_j$ be the vector
in $\CCC^{n-1}$ which is obtained from $z$ by erasing its $j$th coordinate. Observe that for any $j,
k, l, m \in \{1,\ldots,n\}$ the set $A := \{z = (z_1,\ldots,z_n) \in G_0^{[n]}\dd\
(F_{n-1}(z'_j) - F_{n-1}(z'_k))(z_m - z_l) = (F_{n-1}(z'_l) - F_{n-1}(z'_m))(z_k - z_j)\}$ is
relatively closed in $G_0^{[n]}$. What is more, we deduce from (DD1) and \eqref{eqn:Fk} that $A
\supset \Omega^{(n)}$. (Note that if $j \neq k$ and $(z_1,\ldots,z_n) \in \Omega^{(n)}$, then
$[F_{n-1}(z'_j) - F_{n-1}(z'_k)] / (z_k - z_j) = [\Delta(z'_j)f - \Delta(z'_k)f] / (z_k - z_j) =
\Delta(z_1,\ldots,z_n)f$.) So, the density of $\Omega^{(n)}$ in $G_0^{[n]}$ implies that $A =
G_0^{[n]}$. We conclude that if $(z_1,\ldots,z_n) \in G_0^{[n]}$ and $j, k, l, m \in \{1,\ldots,n\}$
are such that $z_j \neq z_k$ and $z_l \neq z_m$, then
$\frac{F_{n-1}(z'_j) - F_{n-1}(z'_k)}{z_k - z_j} = \frac{F_{n-1}(z'_l) - F_{n-1}(z'_m)}{z_m - z_l}$.
The above property enables us to define properly a function $F\dd G_0^{[n]} \setminus \Theta_n \to
\CCC$ (see \eqref{eqn:diag} for the definition of $\Theta_n$) by the rule:
$$F(z_1,\ldots,z_n) = \frac{F_{n-1}(z'_j) - F_{n-1}(z'_k)}{z_k - z_j}$$
where $j$ and $k$ are chosen so that $z_j \neq z_k$. $F$ is locally Lipschitz---since locally it is
the quotient of two locally Lipschitz functions. And $F$ is symmetric since $F_{n-1}$ is such.
Finally,
\begin{equation}\label{eqn:F}
F(z_1,\ldots,z_n) = \Delta(z_1,\ldots,z_n)f \qquad ((z_1,\ldots,z_n) \in \Omega^{(n)}),
\end{equation}
which follows from \eqref{eqn:Fk}.\par
To end the proof, observe that the functions $F$ and $g_z$'s ($z \in \Omega'$) agree. Hence
it suffices to define $G$ as the intersection of $G_0$ and $\bar{\Omega} \cap \bigcup_{z\in\Omega'}
B(z,\epsi_z)$ and the extension of $\Omega^{(n)} \ni (z_1,\ldots,z_n) \mapsto \Delta(z_1,\ldots,z_n)
\in \CCC$ we search for as the union of $F$ and $g_z$'s (understood as a function on $G^{[n]}$).
\end{proof}

Since the proof of the next result is similar to the above, we skip it.

\begin{pro}{DDC}
If $f \in \DDC^n(\Omega)$, then for any $k \in \{1,\ldots,n+1\}$ the function $\Omega^{(k)} \ni
(z_1,\ldots,z_k) \to \Delta(z_1,\ldots,z_k) f \in \CCC$ extends to a continuous symmetric function
$F_k\dd \Omega^{[k]} \to \CCC$.
\end{pro}

It is an easy exercise that a holomorphic function is of class $\DDC^{\infty}$ (this immediately
follows from \THM{gen-cont} below). Divided differences are also involved in the characterization
of one real variable functions extendable to functions of class $C^k$ given by Whitney \cite{wh2}:

\begin{thm}{whitney2}
Let $\Omega \subset \RRR$ be a closed set and $k \geqsl 0$. A function $f\dd \Omega \to \CCC$
extends to a function of class $C^k$ of $\RRR$ into $\CCC$ iff $f \in \DDC^k(\Omega)$.
\end{thm}

Now we prove a generalization of the easier part of the above result.

\begin{pro}{DD-T}
Each function of class $\DDC^k$ is of class $\TC^k$. Moreover, if $f \in \DDC^k(\Omega)$ and
the functions $F_j\dd \Omega^{[j]} \to \CCC\ (j=1,\ldots,k)$ are as in \textup{\PRO{DDC}}, then
$f^{(j-1)}(z) = (j-1)! \cdot F_j(z,\ldots,z)$ for any $j \in \{1,\ldots,k+1\}$ and $z \in \Omega'$.
\end{pro}
\begin{proof}
Fix $z \in \Omega'$, take arbitrary $w \in \Omega$ and let $z_1,\ldots,z_k \in \Omega \setminus
\{w\}$ be distinct points. Put
$$W(z) = \sum_{j=1}^k F_j(z_1,\ldots,z_j) \prod_{s=1}^{j-1} (z - z_s) + F_{k+1}(z_1,\ldots,z_k,w)
\prod_{s=1}^k (z-z_s).$$
It follows from (DD2) that $W(w) = f(w)$, from which we deduce that
\begin{equation}\label{eqn:aux3}
f(w) = \sum_{j=1}^k F_j(z_1,\ldots,z_j) \prod_{s=1}^{j-1} (w - z_s) + F_{k+1}(z_1,\ldots,z_k,w)
\prod_{s=1}^k (w - z_s).
\end{equation}
Now if $z_j \to z$ for $j \in \{1,\ldots,k\}$, \eqref{eqn:aux3} changes into (thanks
to the continuity of $F_1,\ldots,F_{k+1}$):
\begin{multline*}
f(w) = \sum_{j=1}^k F_j(z,\ldots,z) (w - z)^{j-1} + F_{k+1}(z,\ldots,z,w) (w - z)^k\\
= \sum_{j=1}^{k+1} F_j(z,\ldots,z) (w - z)^{j-1} + [F_{k+1}(z,\ldots,z,w) - F_{k+1}(z,\ldots,z,z)]
(w - z)^k.
\end{multline*}
So, to finish the proof it suffices to define $\tau_z\dd \Omega \to \CCC$ as $\tau_z(w) =
F_{k+1}(z,\ldots,z,w) - F_{k+1}(z,\ldots,z,z)$ and note that the function $\Omega' \times \Omega \ni
(z,w) \mapsto \tau_z(w) \in \CCC$ is continuous.
\end{proof}

As a simple consequence of the above result, we obtain

\begin{cor}{DDB}
Let $\Omega \subset \RRR$ be a subinterval or an open subset of the real line and $k \geqsl 1$.
A function $f\dd \Omega \to \CCC$ is of class $\DDB^k$ iff $f$ is of class $C^{k-1}$ and $f^{(k-1)}$
is locally Lipschitz.
\end{cor}
\begin{proof}
First assume $f$ is of class $C^{k-1}$ and $f^{(k-1)}$ is locally Lipschitz. Fix $a \in \Omega$ and
let $\epsi > 0$ and $M$ be such that $|f^{(k-1)}(x) - f^{(k-1)}(y)| \leqsl M |x - y| $ for any $x, y
\in \Omega \cap [a - \epsi,a + \epsi]$ and this last set is an interval. Let $x_1,\ldots,x_{k+1}$ be
distinct points of $\Omega \cap [a - \epsi,a + \epsi]$. It suffices to check that
$|\Delta(x_1,\ldots,x_{k+1}) u| \leqsl M$ for $u \in \{\RE f,\IM f\}$. So, we may assume $f$ is
real-valued. Moreover, thanks to property (DD1) we may also assume that $x_1 < \ldots < x_{k+1}$.
It follows from the mean value theorem for divided differences (see e.g. the argument on page~369
in \cite{wh2}) that there are $\xi \in (x_1,x_k)$ and $\eta \in (x_2,x_{k+1})$ such that
$\Delta(x_1,\ldots,x_k) f = f^{(k-1)}(\xi)$ and $\Delta(x_2,\ldots,x_{k+1}) f = f^{(k-1)}(\eta)$.
Finally, observe that then
$$|\Delta(x_1,\ldots,x_{k+1}) f| = \frac{|f^{(k-1)}(\eta) - f^{(k-1)}(\xi)|}{x_{k+1} - x_1} \leqsl
M \cdot \frac{|\eta - \xi|}{x_{k+1} - x_1} \leqsl M.$$
Conversely, if $f \in \DDB^k(\Omega)$, then $f \in \DDC^{k-1}(\Omega)$ (by \PRO{B-C}) and hence
$f$ is of class $\TC^{k-1}$ (cf. \PRO{DD-T}). Consequently, $f$ is of class $C^{k-1}$, by Whitney's
theorem (\PRO{whitney1}). Finally, we infer from Propositions~\ref{pro:DD-T} and \ref{pro:B-C} that
$f^{(k-1)}$ is locally Lipschitz.
\end{proof}

\begin{rem}{hered}
Using similar arguments as those in the proof of \PRO{DD-T}, one may check that if a function $f\dd
\Omega \to \CCC$ is of class $\DDC^k$ ($k = 1,2,\ldots,\infty$), then for any $j < k$ ($j \geqsl 1$)
the function $f^{(j)}\dd \Omega' \to \CCC$ is of class $\TC^{k-j}$ and $(f^{(j)})^{(s)} =
(f^{(j+s)})\bigr|_{\Omega''}$ (where $\Omega'' = (\Omega')'$) for $s = 1,\ldots,k-j$.
\end{rem}

\SECT{Continuity of functional calculus}

Let us begin this section with a reminder that for totally arbitrary function $f\dd \Omega \to \CCC$
the function $f_{\op}\dd \ddD_k(\Omega) \to \ddD_k(\CCC)$ is well defined for any $k$
by the following rule: $f[D] = W[D]$ where $D$ is a diagonalizable square matrix and $W$ is any
polynomial such that $f\bigr|_{\spc(D)} = W\bigr|_{\spc(D)}$. We are interested in those functions
$f$ for which $f_{\op}$ is continuous on $\ddD_k(\Omega)$ (for fixed $k$). Observe that $f_{\op}\dd
\ddD_1(\Omega) \to \mmM_1(\CCC)$ may naturally be identified with $f\dd \Omega \to \CCC$ and hence
for $k=1$ the characterization is trivial ($f_{\op}\dd \ddD_1(\Omega) \to \mmM_1(\CCC)$ is
continuous iff $f$ is such). Therefore everywhere below we will assume that $k > 1$.\par
In what follows, we shall use the following concept, very often practiced. For a $k \times k$ matrix
$A$ we write $\spc(A) = \{\lambda_1,\ldots,\lambda_k\}$ iff the characteristic polynomial $W_A(z) =
\det(zI - A)$ of $A$ has the form $W_A(z) = (z-\lambda_1)\cdot\ldots\cdot(z-\lambda_k)$. (So,
in the unordered $k$-tuple $\{\lambda_1,\ldots,\lambda_k\}$ the number of appearances of each
of eigenvalues of $A$ coincides with its \textit{algebraic multiplicity}, i.e. its multiplicity
as a root of $W_A$.)\par
The following is a well-known result (see e.g. \cite[Theorem~3.1.2]{ser}). We will apply it
in the proof of \THM{gen-cont} below.

\begin{pro}{conv-spc}
If $k \times k$ matrices $A_1,A_2,A_3,\ldots$ converge to a matrix $A$ and $\sigma(A) =
\{\lambda^{(1)},\ldots,\lambda^{(k)}\}$, then there are scalar sequences
$(\lambda^{(1)}_n)_{n=1}^{\infty}$,\ldots,$(\lambda^{(k)}_n)_{n=1}^{\infty}$ such that
$\lim_{n\to\infty} \lambda^{(j)}_n = \lambda^{(j)}$ for $j=1,\ldots,k$ and $\sigma(A_n) =
\{\lambda^{(1)}_n,\ldots,\lambda^{(k)}_n\}$ for any $n \geqsl 1$.
\end{pro}

The proof of next simple result is left to the reader.

\begin{lem}{closure}
The closure of $\ddD_k(\Omega)$ in $\mmM_k(\CCC)$ coincides with $\mmM_k^o(\bar{\Omega})$.
\end{lem}

The above lemma explains the role played by the set $\mmM_k^o(\Omega)$ and shows that this set
appears quite naturally in topological aspects.\par
Now we are ready to state and prove the main result of the section.

\begin{thm}{gen-cont}
For an arbitrary function $f\dd \Omega \to \CCC$ and $k \geqsl 2$ \tfcae
\begin{enumerate}[\upshape(i)]
\item for any $\lambda \in \Omega'$ there are positive real numbers $\epsi = \epsi(\lambda)$ and
   $M = M(\lambda)$ such that $\|f[X]\| \leqsl M$ provided $X \in \ddD_k(\Omega)$ is such that
   $\|X - \lambda I\| \leqsl \epsi$;
\item $f_{\op}\dd \ddD_k(\Omega) \to \ddD_k(\CCC)$ is continuous;
\item $f$ is of class $\TC^{k-2}$ and $f_{\op}\dd \mmM_k^o(\Omega) \setminus \zzZ_k \to
   \mmM_k(\CCC)$ is continuous;
\item $f \in \DDB^{k-1}(\Omega)$.
\end{enumerate}
Moreover, if condition \textup{(ii)} is fulfilled, then $f$ extends to a function $\tilde{f}\dd
\tilde{\Omega} \to \CCC$ with $\tilde{\Omega} \supset \Omega$ locally compact such that
$\tilde{f}_{\op}\dd \ddD_k(\tilde{\Omega}) \to \ddD_k(\CCC)$ is continuous.
\end{thm}
\begin{proof}
First assume that $f \in \DDB^{k-1}(\Omega)$. By \PRO{B-C}, there exist a locally compact set
$\tilde{\Omega}$, $\Omega \subset \tilde{\Omega} \subset \bar{\Omega}$, and locally Lipschitz
functions $F_j\dd \tilde{\Omega}^{[j]} \to \CCC$ ($j=1,\ldots,k-1$) such that
\begin{equation}\label{eqn:aux30}
F_j(z_1,\ldots,z_j) = \Delta(z_1,\ldots,z_j)f \qquad ((z_1,\ldots,z_j) \in \Omega^{(j)})
\end{equation}
for any $j \in \{1,\ldots,k-1\}$. Define $\tilde{f}\dd \tilde{\Omega} \to \CCC$ as $F_1$. It follows
from \eqref{eqn:aux30} that $\tilde{f}$ extends $f$. What is more, since $\tilde{f}$ and $F_j$'s are
continuous and $\Omega$ is dense in $\tilde{\Omega}$, we conclude from \eqref{eqn:aux30} that also
\begin{equation}\label{eqn:tilde}
F_j(z_1,\ldots,z_j) = \Delta(z_1,\ldots,z_j)\tilde{f} \qquad ((z_1,\ldots,z_j) \in
\tilde{\Omega}^{(j)})
\end{equation}
for $j = 1,\ldots,k-1$. Since $F_{k-1}$ is locally Lipschitz and symmetric, \eqref{eqn:tilde}
implies that $\tilde{f} \in \DDB^{k-1}(\tilde{\Omega})$. Consequently, thanks
to Propositions~\ref{pro:B-C} and \ref{pro:DD-T}, $\tilde{f} \in \TC^{k-2}(\tilde{\Omega})$ and
\begin{equation}\label{eqn:deriv}
F_{j+1}(z,\ldots,z) = \frac{\tilde{f}^{(j)}(z)}{j!}  \qquad (z \in \tilde{\Omega}',\ j \in
\{0,\ldots,k-2\}).
\end{equation}
We will now show that $\tilde{f}_{\op}\dd \mmM_k^o(\tilde{\Omega}) \setminus \zzZ_k \to
\mmM_k(\CCC)$ is continuous (from which one infers (iii)). For simplicity, we put $G :=
\mmM_k^o(\tilde{\Omega}) \setminus \zzZ_k$. Recall that $\tilde{f}[X]$ makes sense for any $X \in G$
since $X \notin \zzZ_k$ (which means that the algebraic multiplicity of any eigenvalue of $X$ is
less than $k$) and $\tilde{f}$ is of class $\TC^{k-2}$.\par
Let matrices $X_1,X_2,\ldots \in G$ converge to $X_0 \in G$.\par
First assume that each of $X_n$ for $n > 0$ is diagonalizable. Write $\spc(X_0) = \{\lambda^{(1)}_0,
\ldots,\lambda^{(k)}_0\}$. \PRO{conv-spc} enables us to write (for each $n \geqsl 1$) $\spc(X_n) =
\{\lambda^{(1)}_n,\ldots,\lambda^{(k)}_n\}$ in a way such that
\begin{equation}\label{eqn:conv-eigen}
\lim_{n\to\infty} \lambda^{(j)}_n = \lambda^{(j)}_0 \qquad (j=1,\ldots,k).
\end{equation}
Further, for each $n > 0$ write $\mu_{X_n}$ in the form $\mu_{X_n}(z) = \prod_{j=1}^k
(z - \lambda^{(j)}_n)^{\nu_n(j)}$ where $\nu_n(j) \in \{0,1\}$. After passing to a subsequence and
rearranging the eigenvalues, we may assume that for some $s \in \{1,\ldots,k\}$ one has $\nu_n(j) =
1$ for $j \leqsl s$ and $\nu_n(j) = 0$ for $j > s$ (for any $n > 0$). Since $\lim_{n\to\infty}
\mu_{X_n}[X_n] = \lim_{n\to\infty} \prod_{j=1}^s (X_n - \lambda^{(j)}_n I) = \prod_{j=1}^s
(X_0 - \lambda^{(j)}_0 I)$, we conclude that
\begin{equation}\label{eqn:mu0}
\prod_{j=1}^s (X_0 - \lambda^{(j)}_0 I) = 0
\end{equation}
and thus the sets $\spc(X_0)$ and $\{\lambda^{(j)}_0\dd\ j \leqsl s\}$ coincide. Further, we infer
from the diagonalizability of $X_n$ ($n > 0$) and the formula for $\mu_{X_n}$ that
\begin{equation}\label{eqn:differ}
\lambda^{(1)}_n,\ldots,\lambda^{(s)}_n \textup{ are different for each } n > 0.
\end{equation}
So, \LEM{closure} combined with \eqref{eqn:conv-eigen} and \eqref{eqn:differ} yields
\begin{equation}\label{eqn:zero}
(\lambda^{(1)}_0,\ldots,\lambda^{(s)}_0) \in \tilde{\Omega}^{[s]}.
\end{equation}
Further, for $n > 0$ put
\begin{equation}\label{eqn:Vn}
V_n(z) = \sum_{q=1}^s \bigl[\Delta(\lambda^{(1)}_n,\ldots,\lambda^{(q)}_n)\tilde{f} \cdot
\prod_{j<q} (z - \lambda^{(j)}_n)\bigr].
\end{equation}
It follows from (DD2) that $V_n(\lambda^{(j)}_n) = \tilde{f}(\lambda^{(j)}_n)$ for $j \leqsl s$ and
hence $\tilde{f}[X_n] = V_n[X_n]$. So, we need to show that $\lim_{n\to\infty} V_n[X_n] =
\tilde{f}[X_0]$. To this end, we consider three cases.\par
First assume that $\card(\spc(X_0)) = 1$. For simplicity, denote by $\lambda$ the unique element
of $\spc(X_0)$. Then $\lim_{n\to\infty} \lambda^{(j)}_n = \lambda$ for $j \in \{1,\ldots,s\}$.
The fact that $X_0 \notin \zzZ_k$ combined with \eqref{eqn:zero} gives
\begin{equation}\label{eqn:zero'}
(X_0 - \lambda I)^{\alpha} = 0 \qquad \textup{where } \alpha := \min(k-1,s)
\end{equation}
and therefore
\begin{equation}\label{eqn:f(X0)}
\tilde{f}[X_0] = \sum_{j=0}^{\alpha-1} \frac{\tilde{f}^{(j)}(\lambda)}{j!} (X_0 - \lambda I)^j.
\end{equation}
We obtain from \eqref{eqn:tilde} and \eqref{eqn:Vn} that
$$V_n[X_n] = \sum_{q=1}^{\alpha} F_q(\lambda^{(1)}_n,\ldots,\lambda^{(q)}_n) \cdot \prod_{j<q}
(X_n - \lambda^{(j)}_n I)$$
provided $s < k$ and
$$V_n[X_n] = \sum_{q=1}^{k-1} F_q(\lambda^{(1)}_n,\ldots,\lambda^{(q)}_n) \cdot \prod_{j<q}
(X_n - \lambda^{(j)}_n I) + \Delta(\lambda^{(1)}_n,\ldots,\lambda^{(k)}_n)\tilde{f} \cdot
\prod_{q=1}^{k-1} (X_n - \lambda^{(j)}_n I)$$
otherwise. Observe that the last summand in the latter formula tends to $0$ as $n$ tends
to $\infty$, since $\tilde{f} \in \DDB^{k-1}$ and $(X - \lambda I)^{k-1} = 0$
(by \eqref{eqn:zero'}). So, in both the cases we get
$$\lim_{n\to\infty} V_n[X_n] = \sum_{q=1}^{\alpha} F_q(\lambda,\ldots,\lambda) \cdot \prod_{j<q}
(X_0 - \lambda I) = \sum_{j=0}^{\alpha-1} F_{j+1}(\lambda,\ldots,\lambda) (X_0 - \lambda I)^j.$$
So, relations \eqref{eqn:deriv} and \eqref{eqn:f(X0)} finish the proof in case $\card(\spc(X_0)) =
1$.\par
Now assume that $\card(\spc(X_0)) > 1$. Then $s > 1$. Let $\sigma$ be any permutation
of $\{1,\ldots,s\}$ such that $\sigma(1) \neq \sigma(s)$. Thanks to (DD1) and \eqref{eqn:tilde},
we may transform \eqref{eqn:Vn} into
\begin{multline*}
V_n(z) = \sum_{q=1}^s \Delta(\lambda^{(\sigma(1))}_n,\ldots,\lambda^{(\sigma(q))}_n)\tilde{f} \cdot
\prod_{j<q} (z-\lambda^{(\sigma(j))}_n)\\
= \sum_{q=1}^{s-1} F_q(\lambda^{(\sigma(1))}_n,\ldots,\lambda^{(\sigma(q))}_n) \cdot
\prod_{j<q} (z-\lambda^{(\sigma(j))}_n)\\
+ \frac{F_{s-1}(\lambda^{(\sigma(1))}_n,\ldots,\lambda^{(\sigma(s-1))}_n)
- F_{s-1}(\lambda^{(\sigma(2))}_n,\ldots,\lambda^{(\sigma(s))}_n)}
{\lambda^{(\sigma(1))}_n - \lambda^{(\sigma(s))}_n} \cdot \prod_{j=1}^{s-1} (z-\lambda^{(j)}_n).
\end{multline*}
Consequently,
\begin{multline}\label{eqn:aux31}
\lim_{n\to\infty} V_n(z) = V(z)
:= \sum_{q=1}^{s-1} F_q(\lambda^{(\sigma(1))}_0,\ldots,\lambda^{(\sigma(q))}_0) \cdot \prod_{j<q}
(z - \lambda^{(\sigma(j))}_0)\\
+ \frac{F_{s-1}(\lambda^{(\sigma(1))}_0,\ldots,\lambda^{(\sigma(s-1))}_0)
- F_{s-1}(\lambda^{(\sigma(2))}_0,\ldots,\lambda^{(\sigma(s))}_0)}
{\lambda^{(\sigma(1))}_0 - \lambda^{(\sigma(s))}_0} \cdot \prod_{j=1}^{s-1} (z - \lambda^{(j)}_0)
\end{multline}
(the above formula means, in particular, that the formula for $V$ is independent of the permutation
$\sigma$) and $\lim_{n\to\infty} V_n[X_n] = V[X_0]$. So, the proof will be completed if we show that
$V[X_0] = \tilde{f}[X_0]$. To do this, it suffices to check that $V^{(q)}(\lambda^{(j)}_0) =
\tilde{f}^{(q)}(\lambda^{(j)}_0)$ for any $j \in \{1,\ldots,s\}$ and $q$ less than the multiplicity
of $\lambda^{(j)}_0$ as a root of $\mu_{X_0}$. To this end, fix $j_0 \in \{1,\ldots,s\}$ and take
any permutation $\tau$ of $\{1,\ldots,s\}$ such that for some $p \in \{1,\ldots,s-1\}$,
$\lambda^{(\tau(1))}_0 = \ldots = \lambda^{\tau(p)}_0 = \lambda^{(j_0)}_0$ and $\lambda^{\tau(q)}_0
\neq \lambda^{(j_0)}_0$ for $q > p$. (Note that $p < s$ because $\card(\spc(X_0)) > 1$.) For
simplicity, put $w := \lambda^{(j_0)}_0$. We infer from \eqref{eqn:mu0} that then the multiplicity
of $w$ as a root of $\mu_{X_0}$ is not greater than $p$. Thus, we only need to check that
$V^{(q)}(w) = \tilde{f}^{(q)}(w)$ for $q = 0,\ldots,p-1$. Substituting in \eqref{eqn:aux31} $\tau$
for $\sigma$, we see that for a suitable polynomial $Q$ one has
$$V(z) = \sum_{q=0}^{p-1} F_{q+1}(w,\ldots,w) (z-w)^q + (z-w)^p Q(z).$$
The above, combined with \eqref{eqn:deriv}, completes the proof.\par
Now assume that $X_n$'s (for $n > 0$) are arbitrary. Fix for a moment $m > 0$. \LEM{closure} implies
that there is a sequence $Y_1,Y_2,\ldots \in \ddD_k(\tilde{\Omega})$ which converges to $X_m$.
It follows from the first part of the proof that $\lim_{n\to\infty} \tilde{f}[Y_n] = \tilde{f}[X_m]$
and thus there is $\beta_m \in \{1,2,3,\ldots\}$ such that for $X'_m := Y_{\beta_m}$ one has
\begin{equation}\label{eqn:aux32}
\|X'_m - X_m\| \leqsl \frac1m \qquad \textup{and} \qquad \|\tilde{f}[X'_m] - \tilde{f}[X_m]\| \leqsl
\frac1m.
\end{equation}
The former inequality in \eqref{eqn:aux32} shows that $\lim_{n\to\infty} X'_n = X_0$. Since $X'_n
\in \ddD_k(\tilde{\Omega})$, we conclude from the first part of the proof that $\lim_{n\to\infty}
\tilde{f}[X'_n] = \tilde{f}[X_0]$ which, combined with the latter inequality in \eqref{eqn:aux32},
yields $\lim_{n\to\infty} \tilde{f}[X_n] = \tilde{f}[X_0]$. This finishes the proof of implication
`(iv)$\implies$(iii)'.\par
Since implications `(ii)$\implies$(i)' and `(iii)$\implies$(ii)' are trivial, we only need to show
that (iv) follows from (i). To this end, fix $\lambda \in \Omega'$ and for any $\ell =
(\lambda_1,\ldots,\lambda_k) \in \Omega^{(k)}$ and a positive real number $\epsi$ denote
by $A_{\ell,\epsi}$ the matrix $[a_{p,q}]$ such that $a_{j,j} = \lambda_j$ ($j = 1,\ldots,k$),
$a_{j+1,j} = \epsi$ ($j = 1,\ldots,k-1$) and $a_{p,q} = 0$ otherwise. Notice that $A_{\ell,\epsi}
\in \ddD_k(\Omega)$, $\spc(A_{\ell,\epsi}) = \{\lambda_1,\ldots,\lambda_k\}$ and $A_{\ell,\epsi} \to
\lambda I$ provided $\epsi \to 0$ and $\ell \stackrel{\Omega^{(k)}}{\to} (\lambda,\ldots,\lambda)
\in \CCC^k$. Let $b_{\ell,\epsi}$ stand for the bottom left corner of $f[A_{\ell,\epsi}]$. So,
if (i) is fulfilled, there are $\delta > 0$ and $C \in \RRR$ such that
\begin{equation}\label{eqn:aux33}
|b_{\ell,\epsi}| \leqsl C \quad \textup{whenever } \epsi < \delta \textup{ and } \ell \in
\prod_{j=1}^k B(\lambda,\delta)
\end{equation}
(for the definition of $B(\lambda,\delta)$ see \eqref{eqn:ball}). Observe that $f[A_{\ell,\epsi}] =
V_{\ell,\epsi}[A_{\ell,\epsi}]$ where, for $\ell = (\lambda_1,\ldots,\lambda_k)$,
$$V_{\ell,\epsi}(z) = \sum_{j=1}^k \Delta(\lambda_1,\ldots,\lambda_j)f \cdot \prod_{s<j}
(z - \lambda_s).$$
Write $V_{\ell,\epsi}$ in the form $V_{\ell,\epsi}(z) = \sum_{j=0}^{k-1} \alpha_j z^j$ and note that
$\alpha_{k-1} = \Delta(\lambda_1,\ldots,\lambda_k)f$. Taking this into account, one may check that
$b_{\ell,\epsi} = \Delta(\lambda_1,\ldots,\lambda_k)f \cdot \epsi^{k-1}$ (compare with the Opitz
formula \cite{opz} or \cite[Proposition~25]{cdb}). Now observe that \eqref{eqn:aux33} holds iff
$\limsup_{\ell\stackrel{\Omega^{(k)}}{\to}(\lambda,\ldots,\lambda)} |\Delta(\ell)f| < \infty$, which
finishes the proof.
\end{proof}

\begin{rem}{distinct}
One may easily conclude from the above proof that if only $f$ is continuous, then $f_{\op}\dd
\ddD_k(\Omega) \to \ddD_k(\CCC)$ is continuous at each $X_0 \in \ddD_k(\Omega)$ with
$\card(\spc(X_0)) = k$ (for any $k$).
\end{rem}

\THM{gen-cont} shows that whenever $f_{\op}\dd \ddD_k(\Omega) \to \mmM_k(\CCC)$ is continuous,
it extends to a continuous function of $\mmM^o_k(\Omega) \setminus \zzZ_k$ into $\mmM_k(\CCC)$.
Taking this into account, it seems to be interesting the question of when $f_{\op}$ extends
to a continuous function of $\mmM^o(\Omega)$. A full answer to this problem gives

\begin{pro}{strong-cont}
For a function $f\dd \Omega \to \CCC$ and arbitrary $k \geqsl 2$ \tfcae
\begin{enumerate}[\upshape(i)]
\item $f_{\op}\dd \ddD_k(\Omega) \to \mmM_k(\CCC)$ extends to a continuous function
   of $\mmM^o_k(\Omega)$;
\item $f$ is of class $\TC^{k-1}$ and $f_{\op}\dd \mmM^o_k(\Omega) \to \mmM_k(\CCC)$ is continuous;
\item $f \in \DDC^{k-1}(\Omega)$.
\end{enumerate}
\end{pro}
\begin{proof}
As in the previous proof, first we assume that $f \in \DDC^{k-1}(\Omega)$. Our aim is to show (ii).
We know from \PRO{DD-T} and \THM{gen-cont} that $f \in \TC^{k-1}(\Omega)$. Fix arbitrary $X_0 \in
\mmM^o_k(\Omega)$ and let matrices $X_1,X_2,\ldots \in \mmM^o_k(\Omega)$ converge to $X_0$.
The argument presented in the last part of the proof of implication `(iv)$\implies$(iii)'
in \THM{gen-cont} ensures us that we may assume each of $X_n$ ($n > 0$) is diagonalizable. Then
if $X_0 \notin \zzZ_k$, we deduce from \THM{gen-cont} that $\lim_{n\to\infty} f[X_n] = f[X_0]$.
Hence we may and do assume that $X_0 \in \zzZ_k$. Let $F_j\dd \Omega^{[j]} \to \CCC$
($j=1,\ldots,n$) be as in \PRO{DDC}. Denote by $\lambda$ the unique element of $\spc(X_0)$. Write
$\spc(X_n) = \{\lambda^{(1)}_n,\ldots,\lambda^{(k)}_n\}$ ($n > 0$). Then $\lim_{n\to\infty}
\lambda^{(j)}_n = \lambda$ ($j=1,\ldots,k$). Mimicing the proof of \THM{gen-cont}, we may assume
that for some $s \in \{1,\ldots,k\}$, $\mu_{X_n}(z) = \prod_{j=1}^s (z - \lambda^{(j)}_n)$ for any
$n > 0$. Observe that $\lim_{n\to\infty} \mu_{X_n}[X_n] = (X_0 - \lambda I)^s$, from which we infer
that $s = k$ (since $X_0 \in \zzZ_k$). This implies that $\card(\spc(X_n)) = k$ for positive $n$ and
thus:
\begin{multline*}
f[X_n] = \sum_{q=1}^k \Delta(\lambda^{(1)}_n,\ldots,\lambda^{(q)}_n)f \cdot \prod_{j<q}
(X_n - \lambda^{(j)} I)\\= \sum_{q=1}^k F_q(\lambda^{(1)}_n,\ldots,\lambda^{(q)}_n) \cdot
\prod_{j<q} (X_n - \lambda^{(j)} I).
\end{multline*}
Consequently, $\lim_{n\to\infty} f[X_n] = \sum_{q=1}^k F_q(\lambda,\ldots,\lambda) \cdot
(X_0 - \lambda I)^{q-1}$. Now an application of \PRO{DD-T} allows us to transform the last equality
into $\lim_{n\to\infty} f[X_n] = \sum_{q=1}^k \frac{f^{(q-1)}(\lambda)}{(q-1)!}
(X_0 - \lambda I)^{q-1} = f[X_0]$ and we are done.\par
Since (i) is readily implied by (ii), it remains to show that (iii) follows from (i). To this end,
fix $\lambda \in \Omega'$. Denote by $A$ the matrix $[a_{p,q}]$ such that $a_{p,q} = 1$ when $p =
q+1$ and $a_{p,q} = 0$ otherwise. For $\ell \in \Omega^{(k)}$ and $\epsi > 0$ let $A_{\ell,\epsi}$
and $b_{\ell,\epsi}$ be as in the proof of \THM{gen-cont}. Note that $A + \lambda I \in
\mmM^o_k(\Omega)$, $A_{\ell,\epsi} \in \ddD_k(\Omega)$ and $A_{\ell,1+\epsi}$ tends to $A+\lambda I$
as
\begin{equation}\label{eqn:aux34}
\epsi \to 0 \quad \textup{and} \quad \ell \stackrel{\Omega^{(k)}}{\to} (\lambda,\ldots,\lambda) \in
\CCC^k.
\end{equation}
We deduce from (i) that $f[A_{\ell,1+\epsi}]$ converges when \eqref{eqn:aux34} holds. Consequently,
$b_{\ell,1+\epsi}$ converges as well. But $b_{\ell,1+\epsi} = \Delta(\ell)f \cdot (1+\epsi)^{k-1}$
and therefore $\Delta(\ell)f$ has a finite limit as $\ell \stackrel{\Omega^{(k)}}{\to}
(\lambda,\ldots,\lambda)$, which finishes the proof.
\end{proof}

Now we are ready to give

\begin{proof}[Proof of \THM{main}]
Let us start with (A). It follows from \THM{gen-cont} and \PRO{strong-cont} that each
of the conditions (i) and (ii) (in point (A)) implies that $f \in \DDC^1(\Omega)$ (recall that
$k \geqsl 3$ and take into account \PRO{B-C}). So, we infer from \PRO{DD-T} that $f \in
\TC^1(\Omega)$ and consequently $f$ is holomorphic, since $\Omega$ is open. Conversely, if $f$ is
holomorphic, then $f_{\op}\dd \mmM_j(\Omega) \to \mmM_j(\CCC)$ is holomorphic for any $j$ as well,
which is readily followed by (i) and (ii).\par
We pass to point (B). It follows from \THM{gen-cont} each of the conditions (i) and (ii) is
equivalent to the fact that $f \in \DDB^{k-1}(\Omega)$ (notice that here $\mmM_k(\Omega) =
\mmM^o_k(\Omega)$). But in these settings this last property is equivalent to (iii), thanks
to \COR{DDB}. Finally, $f_{\op}\dd \ddD_k(\Omega) \to \mmM_k(\CCC)$ extends to a continuous function
of $\mmM_k(\Omega)$ iff $f \in \DDC^{k-1}(\Omega)$ (by \PRO{strong-cont}) or, equivalently, iff
$f \in C^{k-1}(\Omega)$ (see \PRO{whitney1} and \THM{whitney2} and note that being of class
$\DDC^{k-1}$ is a local property; cf. the proofs of \PRO{DD-T} and \COR{DDB}).
\end{proof}

\begin{exm}{TC-dis}
Let us show that being of class $\TC^{\infty}$ is insufficient for the continuity of the extended
functional calculus. Let $\Omega = \{0\} \cup \{1/n\dd\ n \geqsl 2\} \cup \{1/n + 3^{-n}\dd\
n \geqsl 2\}$ and let $f\dd \Omega \to \CCC$ be defined as follows: $f(0) = f(1/n) = 0$ and
$f(1/n + 3^{-n}) = 2^{-n}$ for each $n > 0$. Notice that $\Omega$ is compact and
$\lim_{x\stackrel{\Omega}{\to}0} f(x)/x^k = 0$ for any $k$. This yields that $f \in
\TC^{\infty}(\Omega)$ (indeed, $\Omega' = \{0\}$ and $f^{(k)}(0) = 0$ for each $k$). However,
$f \notin \DDB^1(\Omega)$ because $\Delta(1/n + 3^{-n},1/n)f = (3/2)^n$. So, $f_{\op}\dd
\ddD_k(\Omega) \to \mmM_k(\CCC)$ is discontinuous for any $k > 1$ (by \THM{gen-cont}).\par
The above example shows also that the geometric shape of the set $\Omega$ matters when compering
different concepts of differentiability.
\end{exm}

\SECT{Aspects of uniform continuity}

When dealing with functional calculus for matrices, the adjective `uniform' may refer to two
different aspects of uniformity, namely:
\begin{itemize}
\item `uniform continuity' of $f_{\op}\dd \ddD_k(\Omega) \to \mmM_k(\CCC)$ as independent
   of the point at which the continuity is investigated; that is: for any $\epsi > 0$ there is
   $\delta > 0$ such that $\|f[X] - f[Y]\| \leqsl \epsi$ whenever $X, Y \in \ddD_k(\Omega)$ are such
   that $\|X - Y\| \leqsl \delta$ (here $k$ is fixed);
\item `uniform continuity' of $f_{\op}\dd \ddD_k(\Omega) \to \mmM_k(\CCC)$ as independent of $k$;
   for example: for any $\lambda \in \Omega'$ and each $\epsi > 0$ there is $\delta > 0$ such that
   $\|f[X] - f[\lambda I]\| \leqsl \epsi$ provided $X \in \ddD_n(\Omega)$ is such that
   $\|X - \lambda I\| \leqsl \delta$ (here $n \geqsl 1$ is arbitrary).
\end{itemize}

In this section we discuss both the above approaches. We begin with a useful

\begin{pro}{ext}
For two functions $f\dd \Omega \to \CCC$ and $F\dd \mmM^o_k(\tilde{\Omega}) \to \mmM_k(\CCC)$ where
$\Omega \subset \tilde{\Omega} \subset \bar{\Omega}$ \tfcae
\begin{enumerate}[\upshape(i)]
\item $F$ is continuous and $F(X) = f[X]$ for $X \in \ddD_k(\Omega)$;
\item $f$ extends to a function $\tilde{f}\dd \tilde{\Omega} \to \CCC$ of class $\DDC^{k-1}$ such
   that $F(X) = \tilde{f}[X]$ for any $X \in \ddD_k(\Omega)$.
\end{enumerate}
\end{pro}
\begin{proof}
We conclude from \PRO{strong-cont} that (i) follows from (ii). Now assume (i) holds and observe that
for any $\lambda \in \Omega$, $F(\lambda I) = f(\lambda) I$. So, it follows from the continuity
of $F$ and the density of $\Omega$ in $\tilde{\Omega}$ that for and $\lambda \in \tilde{\Omega}$,
$F(\lambda I)$ is a scalar multiple of the identity matrix $I$. This notice enables us to define
a function $\tilde{f}\dd \tilde{\Omega} \to \CCC$ by the formula $F(\lambda I) = \tilde{f}(\lambda)
I$. It is clear that $\tilde{f}$ extends $f$ and is continuous. Fix $X \in \ddD_k(\Omega)$ and take
an invertible $k \times k$ matrix $P$ such that $D := PXP^{-1}$ is diagonal. It is easily seen that
there is a sequence $D_1,D_2,\ldots$ of diagonal matrices belonging to $\ddD_k(\Omega)$ which
converges to $D$. It follows from the continuity of $\tilde{f}$ that $\lim_{n\to\infty}
\tilde{f}[D_n] = \tilde{f}[D]$ and thus
\begin{multline*}
\tilde{f}[X] = P^{-1} \tilde{f}[D] P = \lim_{n\to\infty} (P^{-1} f[D_n] P) = \lim_{n\to\infty}
f[P^{-1} D_n P]\\= F(P^{-1} D P) = F(X).
\end{multline*}
So, we see that $\tilde{f}_{\op}\dd \ddD_k(\tilde{\Omega}) \to \mmM_k(\CCC)$ extends to a continuous
function of $\mmM^o_k(\tilde{\Omega})$ into $\mmM_k(\CCC)$ (namely, $F$). Consequently, $\tilde{f}
\in \DDC^{k-1}(\tilde{\Omega})$, by \PRO{strong-cont}. Then also $\tilde{f}_{\op}\dd
\mmM^o_k(\tilde{\Omega}) \to \mmM_k(\CCC)$ is continuous and therefore $F(X) = \tilde{f}[X]$ for any
$X \in \mmM^o_k(\tilde{\Omega})$, since these two functions coincide on a dense set
(cf. \LEM{closure}).
\end{proof}

First we shall characterize those functions $f\dd \Omega \to \CCC$ for which $f_{\op}\dd
\ddD_2(\Omega) \to \mmM_2(\Omega)$. As the following result shows, this characterization
(in general) is somewhat strange.

\begin{lem}{2}
For a function $f\dd \Omega \to \CCC$ \tfcae
\begin{enumerate}[\upshape(i)]
\item $f_{\op}\dd \ddD_2(\Omega) \to \mmM_2(\CCC)$ is uniformly continuous;
\item $f$ is Lipschitz, the formula
   \begin{equation}\label{eqn:z+w}
   z+w \mapsto f(z)+f(w) \qquad (z,w) \in \Omega^{(2)}
   \end{equation}
   well defines a uniformly continuous function on $\{z+w\dd\ (z,w) \in \Omega^{(2)}\}$, and there
   exists $\delta > 0$ such that
   \begin{equation}\label{eqn:Delta-const}
   \Delta(z,w)f = \Delta(z',w')f \quad \textup{if } (z,w), (z',w') \in \Omega^{(2)}
   \textup{ and } |(z+w) - (z'+w')| \leqsl \delta.
   \end{equation}
\end{enumerate}
\end{lem}
\begin{proof}
First assume $f_{\op}\dd \ddD_2(\Omega) \to \mmM_2(\CCC)$ is uniformly continuous. We want to prove
that all conditions of (ii) are fulfilled. To show that $f$ is Lipschitz, we need to check that
the function $(z,w) \ni \Omega^{(2)} \mapsto \Delta(z,w)f \in \CCC$ is bounded. To this end, fix
a sequence $(z_1,w_1),(z_2,w_2),\ldots$ of arbitrary elements of $\Omega^{(2)}$ and consider
the matrices $A_n = \begin{pmatrix}z_n & 0\\0 & w_n\end{pmatrix}$ and $A_n' = \begin{pmatrix}z_n &
0\\\epsi_n & w_n\end{pmatrix}$ where $\eeE = (\epsi_1,\epsi_2,\ldots)$ is a sequence convergent
to $0$. Then $A_n, A_n' \in \ddD_2(\Omega)$ and $\lim_{n\to\infty} \|A_n - A_n'\| = 0$. So,
we conclude from (i) that the sequence $\gamma(\eeE) = (\gamma_1(\eeE),\gamma_2(\eeE),\ldots)$
of bottom left corners of $f[A_n] - f[A_n']$ is bounded. But $\gamma_n(\eeE) = \Delta(z_n,w_n)
\epsi_n$. Since this sequence is bounded for any $\eeE$ convergent to $0$, we conclude that
$\sup_{n\geqsl1} |\Delta(z_n,w_n)f| < \infty$.\par
Now we claim that there is $\delta > 0$ such that \eqref{eqn:Delta-const} is satisfied. For if not,
there are two sequences $(z_1,w_1),(z_2,w_2),\ldots$ and $(z_1',w_1'),(z_2',w_2'),\ldots$
of elements of $\Omega^{(2)}$ such that
\begin{equation}\label{eqn:aux40}
\epsi_n := (z_n'+w_n') - (z_n+w_n) \to 0 \quad (n\to\infty)
\end{equation}
and $\kappa_n := \Delta(z_n,w_n)f - \Delta(z_n',w_n')f$ is nonzero for each $n$. Fix a sequence
$a_1,a_2,\ldots$ of nonzero complex numbers, put $b_n := \frac{z_n(w_n+\epsi_n) - z_n'w_n'}{a_n}$
and define matrices $A_n$ and $A_n'$ ($n > 0$) as follows: $A_n = \begin{pmatrix}z_n & 0\\a_n & w_n
\end{pmatrix}$, $A_n' = \begin{pmatrix}z_n & b_n\\a_n & w_n+\epsi_n\end{pmatrix}$. It is clear that
$A_n \in \ddD_2(\Omega)$ and $\spc(A_n) = \{z_n,w_n\}$. Observe that $\tr(A_n') = z_n' + w_n'$,
by \eqref{eqn:aux40} (`$\tr(X)$' is the trace of a matrix $X$), and $\det(A_n') = z_n' w_n'$,
by the definition of $b_n$. We conclude that $\spc(A_n') = \{z_n',w_n'\}$ and thus $A_n' \in
\ddD_2(\Omega)$ (since $z_n' \neq w_n'$). Notice that $\lim_{n\to\infty} \|A_n - A_n'\| = 0$ iff
\begin{equation}\label{eqn:aux41}
\lim_{n\to\infty} b_n = 0
\end{equation}
(thanks to \eqref{eqn:aux40}). Further, it follows from the definition of the extended functional
calculus that $f[A_n] = f(z_n) I + \Delta(z_n,w_n)f \cdot (A_n - z_n I)$ and $f[A_n'] = f(z_n') I
+ \Delta(z_n',w_n')f \cdot (A_n' - z_n' I)$ (where $I$ is the $2 \times 2$ unit matrix).
Consequently, if $\gamma_n$ denotes the bottom left corner of $f[A_n] - f[A_n']$, then $\gamma_n =
\kappa_n a_n$. Since $\kappa_n \neq 0$, we see that it is possible to find $a_n$'s such that
\eqref{eqn:aux41} holds and $\lim_{n\to\infty} |\gamma_n| = \infty$. But then, for such $a_n$'s,
$\lim_{n\to\infty} \|A_n - A_n'\| = 0$ and the sequence of $f[A_n] - f[A_n']$ ($n > 0$) is
unbounded, which contradicts (i). This finishes the proof of the last claim in (ii). Additionally,
observe that if we continue the above argument, i.e. starting from \eqref{eqn:aux40}, then choosing
$a_n$'s in a way such that \eqref{eqn:aux41} holds, and defining $A_n$ and $A_n'$ as above, we will
infer from (i) that $\lim_{n\to\infty} \|f[A_n] - f[A_n']\| = 0$. Consequently, $\lim_{n\to\infty}
(\tr(f[A_n]) - \tr(f[A_n'])) = 0$. But $\tr(f[A_n]) = f(z_n)+f(w_n)$ (since $\spc(f[A_n]) =
\{z_n,w_n\}$) and analogously $\tr(f[A_n']) = f(z_n')+f(w_n')$. This shows that if $(z_n,w_n),
(z_n',w_n') \in \Omega^{(2)}$ and $\lim_{n\to\infty} |(z_n+w_n) - (z_n',w_n')| = 0$, then
$\lim_{n\to\infty} |(f(z_n)+f(w_n)) - (f(z_n')+f(w_n'))| = 0$ as well. Equivalently, \eqref{eqn:z+w}
well defines a uniformly continuous function and the proof of (ii) is complete.\par
Now assume that all conditions of (ii) are fulfilled. Fix two sequences $A_1,A_2,\ldots$ and
$A_1',A_2',\ldots$ of members of $\ddD_2(\Omega)$ such that $\lim_{n\to\infty} \|A_n - A_n'\| = 0$.
Our aim is to show that $\lim_{n\to\infty} \|f[A_n] - f[A_n']\| = 0$ as well. To this end,
we consider three cases.\par
If both $A_n$ and $A_n'$ are scalar multiples of the unit matrix, then clearly $\|f[A_n] - f[A_n']\|
\leqsl L(f) \cdot \|A_n - A_n'\|$ where $L(f)$ is a Lipschitz constant for $f$. In that case
the assertion is therefore immediate.\par
If, for example, $A_n'$ is a scalar multiple of the unit matrix, say $A_n' = \gamma_n I$, and $A_n$
is not, then write $\spc(A_n) = \{z_n,w_n\}$ and note that $(z_n,w_n) \in \Omega^{(2)}$ and
$\spc(A_n - A_n') = \{z_n-\gamma_n,w_n-\gamma_n\}$. Consequently, $\lim_{n\to\infty} |z_n-\gamma_n|
= 0$, and $f[A_n] = f(z_n)I + \Delta(z_n,w_n)f \cdot (A_n-z_n I)$ and $f[A_n'] = f(\gamma_n)I$. So,
\begin{multline*}
\|f[A_n] - f[A_n']\| = \|(f(z_n) - f(\gamma_n))I + \Delta(z_n,w_n)f \cdot [(A_n-A_n')
+ (\gamma_n-z_n)I]\|\\ \leqsl 2 L(f) |z_n - \gamma_n| + L(f) \|A_n - A_n'\| \to 0 \qquad
(n\to\infty).
\end{multline*}\par
Finally, we assume that neither $A_n$ nor $A_n'$ is not a scalar multiple of the unit matrix. Then
$\spc(A_n) = \{z_n,w_n\}$ and $\spc(A_n') = (z_n',w_n')$ for some $(z_n,w_n), (z_n',w_n') \in
\Omega^{(2)}$. For simplicity, denote by $F$ the function defined by \eqref{eqn:z+w}. So, $F(z+w) =
f(z) + f(w)$ for any $(z,w) \in \Omega^{(2)}$. Consequently, $F(\tr(A_n)) = \tr(f[A_n])$ and
$F(\tr[A_n']) = \tr(f[A_n'])$. We conclude from the uniform continuity of $F$ that
\begin{equation}\label{eqn:trace}
\lim_{n\to\infty} |\tr(f[A_n]) - \tr(f[A_n'])| = 0
\end{equation}
(because $|\tr(A_n) - \tr(A_n')| \to 0$). Further, it follows from \eqref{eqn:Delta-const} that
$\Delta(z_n,w_n)f = \Delta(z_n',w_n')f$ for almost all $n$. Hence, we may assume that for all $n$,
$\Delta(z_n,w_n)f = \Delta(z_n',w_n')f =: \varrho_n$. Recall that $|\varrho_n| \leqsl L(f)$.
A straightforward calculation shows that $f(z_n) - \Delta(z_n,w_n)f \cdot z_n =
\frac{f(z_n)+f(w_n)}{2} - \Delta(z_n,w_n)f \cdot \frac{z_n+w_n}{2}$ (and similarly for
$(z_n',w_n')$). So, $f[A_n] = f(z_n) I + \Delta(z_n,w_n)f \cdot (A_n - z_n I) = \frac12 (\tr(f[A_n])
- \varrho_n \tr(A_n))I + \varrho_n A_n$ and a similar formula for $f[A_n']$. Finally, taking into
account \eqref{eqn:trace}, we obtain
\begin{multline*}
\|f[A_n]-f[A_n']\| \leqsl \frac12 |\tr(f[A_n])-\tr(f[A_n'])| + |\varrho_n| \cdot
\Bigl(\frac{|\tr(A_n-A_n')|}{2} + \|A_n-A_n'\|\Bigr)\\\leqsl \frac12 |\tr(f[A_n])-\tr(f[A_n'])|
+ 2 L(f) \|A_n-A_n'\| \to 0 \qquad (n\to\infty)
\end{multline*}
which finishes the proof.
\end{proof}

For simplicity, let us call a function $f\dd \Omega \to \CCC$ \textit{affine} (resp. \textit{affine
on a set $A \subset \Omega$}) iff there exist $a, b \in \CCC$ such that $f(z) = az + b$ for any
$z \in \Omega$ (resp. for any $z \in A$).\par
As consequences of \LEM{2}, we obtain the next three results.

\begin{pro}{2}
Let $\Omega$ be a subset of $\CCC$ such that $\card(\Omega) > 1$ and
\begin{equation}\label{eqn:-}
\inf \{|z-w|\dd\ (z,w) \in \Omega^{(2)}\} = 0.
\end{equation}
Then for any function $f\dd \Omega \to \CCC$ \tfcae
\begin{enumerate}[\upshape(i)]
\item $f_{\op}\dd \ddD_2(\Omega) \to \mmM_2(\CCC)$ is uniformly continuous;
\item there exist positive real numbers $\epsi$ and $M$ such that $\|f[X] - f[Y]\| \leqsl M$
   whenever $X, Y \in \ddD_2(\Omega)$ are such that $\|X - Y\| \leqsl \epsi$;
\item $f$ is affine.
\end{enumerate}
\end{pro}
\begin{proof}
We only need to prove that (iii) follows from (ii). Notice that the proof of \LEM{2} shows that
if (ii) is fulfilled, then there is $\delta > 0$ such that \eqref{eqn:Delta-const} holds (see
\LEM{2}). Then it follows from \eqref{eqn:-} that there are $z_0, z_1 \in \Omega$ with $0 <
|z_0 - z_1| < \delta$. Observe that \eqref{eqn:Delta-const} implies that
\begin{equation}\label{eqn:almaff}
\Delta(z_0,z)f = \Delta(z_1,z)f \qquad \textup{for any } z \in \Omega \setminus \{z_0,z_1\}.
\end{equation}
It is an elementary observation that \eqref{eqn:almaff} is equivalent to
$$f(z) = \Delta(z_0,z_1)f \cdot z + \frac{f(z_0)z_1 - f(z_1)z_0}{z_1 - z_0} \qquad (z \in \Omega)$$
which finishes the proof.
\end{proof}

\begin{cor}{2.1}
If $\Omega$ is a subset of $\CCC$ such that for some $\epsi > 0$, $|(z+w) - (z'+w')| \geqsl \epsi$
for any two distinct pairs $(z,w), (z',w') \in \Omega^{(2)}$, then for an arbitrary function $f\dd
\Omega \to \CCC$, $f_{\op}\dd \ddD_2(\Omega) \to \mmM_2(\CCC)$ is uniformly continuous iff $f$ is
Lipschitz.
\end{cor}

The proof of \COR{2.1} is based on \LEM{2} and we leave it to the reader.

\begin{cor}{2.2}
Let $\Omega$ be a subset of $\CCC$ which is symmetric with respect to some $\alpha \in \CCC$; that
is, $2\alpha - z \in \Omega$ for any $z \in \Omega$. If $f_{\op}\dd \ddD_2(\Omega) \to
\mmM_2(\Omega)$ is uniformly continuous \textup{(}where $f\dd \Omega \to \CCC$\textup{)}, then $f$
is affine on $\Omega \setminus \{\alpha\}$. In particular, if $\Omega$ is an additive subgroup
of $\CCC$, then $f_{\op}\dd \ddD_2(\Omega) \to \mmM_2(\CCC)$ is uniformly continuous iff $f$ is
affine.
\end{cor}
\begin{proof}
Assume $f_{\op}\dd \ddD_2(\Omega) \to \mmM_2(\CCC)$ is uniformly continuous. Fix $w \in \Omega
\setminus \{\alpha\}$ and put $m = \Delta(w,2\alpha - w)f$. Note that for any $z \in \Omega
\setminus \{\alpha\}$ we have $(z,2\alpha - z), (w,2\alpha - w) \in \Omega^{(2)}$ and $z + (2\alpha
- z) = w + (2\alpha - w)$. Consequently, $\Delta(z,2\alpha - z)f = m$, thanks
to \eqref{eqn:Delta-const}, and $f(z) + f(2\alpha - z) = f(w) + f(2\alpha - w)$, by \eqref{eqn:z+w}
(see \LEM{2}). We infer from these two equalities that $f(z) = m(z - \alpha) +
\frac{f(w)+f(2\alpha-w)}{2}$ and hence $f$ is affine on $\Omega \setminus \{\alpha\}$.\par
Now if $\Omega$ is a nontrivial additive group, then $0 \in \Omega'$ or there is $x \in \Omega$ such
that $\frac12 x \notin \Omega$. In the first case the assertion follows from \PRO{2}, while
in the second from the above prove (since $\Omega$, being a group, is symmetric with respect
to $\frac12 x$ and $\Omega \setminus \{\frac12 x\} = \Omega$).
\end{proof}

\begin{exm}{2}
As the following somewhat strange example shows, the assertion of \COR{2.2} cannot be strengthened
in general. Let $\Omega = \{k^3\dd\ k \in \ZZZ\}$ and $f\dd \Omega \to \CCC$ be arbitrary. We claim
that $f_{\op}\dd \Omega \to \CCC$ is uniformly continuous iff $f$ is affine on $\Omega \setminus
\{0\}$ (so, $f(0)$ may be chosen independently of other values of $f$). The necessity of the latter
condition follows from \COR{2.2}. To see its sufficiency, we involve \LEM{2}. It is easily seen that
if $f$ is affine on $\Omega \setminus \{0\}$, then automatically $f$ is Lipschitz (on $\Omega$).
So, taking into account \LEM{2}, it is enough to check that if $(k,l), (m,0) \in \Omega^{(2)}$ and
$k + l = m + 0$, then either $k$ or $l$ is zero. But this simply follows from the last Fermat
theorem (for exponent $3$).
\end{exm}

The above results show that the characterization of those functions $f\dd \Omega \to \CCC$ for which
$f_{\op}\dd \ddD_k(\Omega) \to \mmM_k(\CCC)$ is uniformly continuous for $k = 2$ does depend
on the geometry of $\Omega$. The situation changes when $k \geqsl 3$, as shown by \PRO{uniform},
the proof of which we now turn to.

\begin{proof}[Proof of \PRO{uniform}]
Since the case when $k = 2$ and $\bar{\Omega}' \neq \varempty$ follows from \PRO{2}, we may and do
assume that $k \geqsl 3$. We only need to explain why (iii) follows from (i). First of all observe
that $f$ is affine iff $\Delta(x,y,z)f = 0$ for any $(x,y,z) \in \Omega^{(3)}$. For a fixed $(x,y,z)
\in \Omega^{(3)}$ denote by $A(w)$ (where $w \in \CCC$) the $k \times k$ matrix $[a_{p,q}]$ such
that $a_{1,1} = x$, $a_{2,2} = y$, $a_{p,p} = z$ for $p > 2$, $a_{2,1} = a_{3,2} = w$ and $a_{p,q} =
0$ otherwise. It is easy to check that $A(w) \in \ddD_2(\Omega)$, $\spc(A(w)) = \{x,y,z\}$ and
the function $\CCC \ni x \mapsto A(w) \in \mmM_k(\CCC)$ is uniformly continuous. Consequently,
if (i) is satisfied and $b(w)$ denotes the entry of $f[A(w)]$ which lies in its third row and first
column, then there exists $\delta > 0$ such that $|b(w) - b(w')| \leqsl M$ whenever $|w - w'| \leqsl
\delta$ (where $M$ is as in (i)). But $b(w) = \Delta(x,y,z)f \cdot w^2$ and hence the function
$\RRR \ni w \mapsto |b(w) - b(w+\delta)| \in \RRR$ is unbounded unless $\Delta(x,y,z) = 0$. This
finishes the proof.
\end{proof}

\PRO{uniform} says that the problem of characterizing `operator Lipschitz' or `operator H\"{o}lder'
functions, which is very interesting and highly nontrivial for functional calculus for selfadjoint
(or, more generally, normal) matrices, becomes trivial for extended functional calculus. It seems
to be a valid supposition that the main reason for this is that for any compact set $L$ in $\CCC$
having more than one point and each $k > 1$ the closure of $\ddD_k(L)$ is unbounded (and hence
noncompact). Taking this into account, it seems to be reasonable to make some restrictions when
studying uniform continuity of the extended functional calculus. Let us now state two results
in this direction.

\begin{pro}{uni1}
For a function $f\dd \Omega \to \CCC$ and $k \geqsl 2$ \tfcae
\begin{enumerate}[\upshape(i)]
\item there is $\epsi > 0$ such that for any $X_0 \in \ddD_k(\Omega)$, $f_{\op}\dd \ddD_k(\Omega)
   \to \mmM_k(\CCC)$ is uniformly continuous on the set $\{X \in \ddD_k(\Omega)\dd\ \|X - X_0\|
   < \epsi\}$;
\item $f_{\op}\dd \ddD_k(\Omega) \to \mmM_k(\CCC)$ is uniformly continuous on every bounded subset
   of $\ddD_k(\Omega)$;
\item $f$ extends to a function $F \in \DDC^{k-1}(\bar{\Omega})$.
\end{enumerate}
Moreover, if condition \textup{(iii)} holds and $F$ is as there, then the function $F_{\op}\dd
\mmM^o_k(\bar{\Omega}) \to \mmM_k(\CCC)$ is uniformly continuous on every bounded subset
of $\mmM^o_k(\bar{\Omega})$.
\end{pro}
\begin{proof}
Since the closure (in $\mmM_k(\CCC)$) of every bounded subset of $\mmM_k(\CCC)$ is compact, it is
easy to check that (ii) follows from (i). Further, if (ii) is fulfilled, it follows from
the completeness of $\mmM_k(\CCC)$ that $f_{\op}\dd \ddD_k(\Omega) \to \mmM_k(\CCC)$ extends
to a continuous function $G\dd \mmM^o_k(\bar{\Omega}) \to \mmM_k(\CCC)$ (here \LEM{closure} is
applied). Now \PRO{ext} shows that (iii) is satisfied. Finally, assume $F$ is as in (iii). Then
$F_{\op}\dd \mmM^o_k(\bar{\Omega}) \to \mmM_k(\CCC)$ is continous, by \PRO{strong-cont}. What is
more, $\mmM^o_k(\bar{\Omega})$ is closed in $\mmM_k(\CCC)$ (cf. \LEM{closure}) and thus for every
bounded set $L \subset \mmM^o_k(\bar{\Omega})$, $F_{\op}$ is uniformly continuous on $L$ (since
the closure of $L$ in $\mmM^o_k(\bar{\Omega})$ is compact). This proves (i) and completes the proof.
\end{proof}

\begin{pro}{uni2}
For a function $f\dd \Omega \to \CCC$ and $k \geqsl 2$ \tfcae
\begin{enumerate}[\upshape(i)]
\item for any $z \in \Omega$, there is a relatively open \textup{(}in $\ddD_k(\Omega)$\textup{)}
   neighbourhood of $z I$ on which $f_{\op}\dd \ddD_k(\Omega) \to \mmM_k(\CCC)$ is uniformly
   continuous;
\item every point of $\ddD_k(\Omega)$ has a relatively open \textup{(}in $\ddD_k(\Omega)$\textup{)}
   neighbourhood on which the function $f_{\op}\dd \ddD_k(\Omega) \to \mmM_k(\CCC)$ is uniformly
   continuous;
\item $f$ extends to a function $\tilde{f} \in \DDC^{k-1}(\tilde{\Omega})$ for some locally compact
   set $\tilde{\Omega}$ with $\Omega \subset \tilde{\Omega} \subset \bar{\Omega}$.
\end{enumerate}
Moreover, if condition \textup{(iii)} holds and $\tilde{f}$ is as there, then every point
of $\mmM^o_k(\tilde{\Omega})$ has a relatively open
\textup{(}in $\mmM^o_k(\tilde{\Omega})$\textup{)} neighbourhood on which the function
$\tilde{f}_{\op}\dd \mmM^o_k(\tilde{\Omega}) \to \mmM_k(\CCC)$ is uniformly continuous.
\end{pro}
\begin{proof}
First assume (iii) holds and let $\tilde{f}$ be as there. Then, according to \PRO{strong-cont},
$\tilde{f}_{\op}\dd \mmM^o_k(\tilde{\Omega}) \to \mmM_k(\CCC)$ is continuous. Since $\tilde{\Omega}$
is locally compact, there is an open (in $\CCC$) set $U$ such that $\tilde{\Omega} = U \cap
\bar{\Omega}$. Observe that then $\mmM^o_k(\tilde{\Omega}) = \mmM_k(U) \cap \mmM^o_k(\bar{\Omega})$.
But $\mmM_k(U)$ is open in $\mmM_k(\CCC)$ (e.g. by \PRO{conv-spc}), while $\mmM^o_k(\bar{\Omega})$
is closed (\LEM{closure}). So, $\mmM^o_k(\tilde{\Omega})$, being the intersection of an open and
a closed set in the locally compact space $\mmM_k(\CCC)$, is locally compact as well. Consequently,
$\tilde{f}_{\op}$ is locally uniformly continuous, which proves (ii) and the additional claim
of the proposition. Of course, (i) obviously follows from (ii).\par
Now assume (i) is fulfilled. For each $z \in \Omega$ let $\epsi_z > 0$ be such that $f_{\op}$ is
uniformly continuous on $D_z := B_z \cap \ddD_k(\Omega)$ where $B_z := \{X \in \mmM_k(\CCC)\dd\ \|X
- zI\| < \epsi_z\}$. Let $G_z\dd B_z \cap \mmM^o_k(\bar{\Omega}) \to \mmM_k(\CCC)$ denote
the (unique) continuous extension of $f_{\op}\bigr|_{D_z}$. It then follows that the union of all
$G_z$'s is a well defined continuous function on $B \cap \mmM^o_k(\bar{\Omega})$ where $B =
\bigcup_{z\in\Omega} B_z$. Further, let $U = \bigcup_{z\in\Omega} B(z,\epsi_z)\ (\subset \CCC)$ (see
\eqref{eqn:ball}) and $\tilde{\Omega} = U \cap \bar{\Omega}$. Notice that $\tilde{\Omega}$ is
locally compact and $\Omega \subset \tilde{\Omega}$. It is easily seen that $\lambda I \in B$ for
any $\lambda \in \tilde{\Omega}$. Moreover, since $G(\lambda I) =
\lim_{z\stackrel{\Omega}{\to}\lambda} f[z I]$ (by the continuity of $G$), we see that $G(\lambda I)
= w I$ for some $w \in \CCC$. We define a function $\tilde{f}\dd \tilde{\Omega} \to \CCC$
by the rule: $G(\lambda I) = \tilde{f}(\lambda) I\ (\lambda \in \tilde{\Omega})$. It is easily seen
that $\tilde{f}$ extends $f$. Now we shall show that $\tilde{f} \in \DDC^{k-1}(\tilde{\Omega})$,
which will finish the proof.\par
It is clear that $\tilde{f}$ is continuous. We shall apply this fact in the sequel. Let $\lambda \in
\tilde{\Omega}$ and matrices $X_1,X_2,\ldots \in \ddD_k(\tilde{\Omega})$ converge to $\lambda I$.
Let $z \in \Omega$ be such that $\lambda \in B(z,\epsi_z)$. Then $\lambda I \in B_z$ and hence also
$X_n \in B_z$ for almost all $n$. We may assume $X_n \in B_z$ for all $n$. It then follows from
the continuity of $\tilde{f}$ and the density of $\Omega$ in $\tilde{\Omega}$ that for any $n$ one
may find a matrix $X_n' \in B_z \cap \ddD_k(\Omega)$ such that $\|X_n' - X_n\| \leqsl \frac1n$ and
\begin{equation}\label{eqn:aux44}
\|\tilde{f}[X_n'] - \tilde{f}[X_n]\| \leqsl \frac1n
\end{equation}
(use the diagonalizability argument; cf. the proof of \PRO{ext}). But then $\lim_{n\to\infty}
\|X_n' - \lambda I\| = 0$ and hence $\tilde{f}[X_n'] = f[X_n'] = G(X_n') \to G(\lambda I) =
\tilde{f}[\lambda I]$ as $n\to\infty$, which, combined with \eqref{eqn:aux44}, yields
$\lim_{n\to\infty} \tilde{f}[X_n] = \tilde{f}[\lambda I]$. We now infer from \THM{gen-cont} that
$\tilde{f} \in \DDB^{k-1}(\tilde{\Omega})$ and $\tilde{f}_{\op}\dd \mmM^o_k(\tilde{\Omega})
\setminus \zzZ_k \to \mmM_k(\CCC)$ is continuous. In particular,
\begin{equation}\label{eqn:aux45}
\tilde{f}[X] = G(X) \quad \textup{for any } X \in B \cap \mmM^o_k(\tilde{\Omega}) \setminus \zzZ_k.
\end{equation}
To ensure that $\tilde{f}$ is of class $\DDC^{k-1}$, it is enough to check that $\tilde{f}_{\op}\dd
\ddD_k(\tilde{\Omega}) \to \mmM_k(\CCC)$ extends to a continuous function
of $\mmM^o_k(\tilde{\Omega})$ into $\mmM_k(\CCC)$ (according to \PRO{strong-cont}). Equivalently,
we only need to show that if $X_1,X_2,\ldots$ are arbitrary matrices belonging
to $\ddD_k(\tilde{\Omega})$ which converge to some $X \in \mmM^o_k(\tilde{\Omega})$, then
the sequence $\tilde{f}[X_1],\tilde{f}[X_2],\ldots$ has a limit in $\mmM_k(\CCC)$ (such a criterion
for extendability to a continuous function is a general topological fact in metric spaces; see also
the last part of the proof of implication `(iv)$\implies$(iii)' in \THM{gen-cont}). To this end,
assume $X_1,X_2,\ldots \in \ddD_k(\tilde{\Omega})$ converge to $X \in \mmM^o_k(\tilde{\Omega})$.
If $X \notin \zzZ_k$, then $\lim_{n\to\infty} \tilde{f}[X_n] = \tilde{f}[X]$, thanks to the previous
part of the proof. So, we may assume $X \in \zzZ_k$. Denote by $\lambda \in \tilde{\Omega}$
the unique element of $\spc(X)$. Let $z \in \Omega$ be such that $\lambda \in B(z,\epsi_z)$. Then
also $\lambda I \in B_z$. Denote by $A$ the $k \times k$ matrix $[a_{p,q}]$ such that $a_{j+1,j} =
1$ ($j=1,\ldots,k-1$) and $a_{p,q} = 0$ otherwise. Since $B_z$ is open in $\mmM_k(\CCC)$, we see
there is $\delta > 0$ such that $X' := \lambda I + \delta A \in B_z$. Notice that there is
an invertible matrix $P \in \mmM_k(\CCC)$ such that $P X P^{-1} = X'$ (because $X \in \zzZ_k$ and
$\lambda \in \spc(X)$). Then the matrices $X_n' := P X_n P^{-1}$ converge to $X'$ and, consequently,
belong to $B_z$ for almost all $n$. So, $X_n' \in B \cap \mmM^o_k(\tilde{\Omega}) \setminus \zzZ_k$,
$X' \in B \cap \mmM^o_k(\tilde{\Omega})$ and therefore, by \eqref{eqn:aux45} and the continuity
of $G$, $\lim_{n\to\infty} \tilde{f}[X_n'] = G(X')$. Finally, we conclude that $\lim_{n\to\infty}
\tilde{f}[X_n] = \lim_{n\to\infty} P^{-1} \tilde{f}[X_n'] P = P^{-1} G(X') P$ and we are done.
\end{proof}

Now we turn to the second approach to uniform continuity---namely, when the `quality' of continuity
(of the extended functional calculus) is, in a sense, independent of the degree of diagonalizable
matrices. The main result in this topic is \THM{holo}, which we now want to prove.

\begin{proof}[Proof of \THM{holo}]
Of course, (i) obviously follows from (ii). It is also not so difficult, involving holomorphic
functional calculus for bounded Hilbert space operators, that (ii) is implied by (iii). Indeed,
denoting by $I$ the identity operator on $\ell_2$ (= separable infinite-dimensional complex Hilbert
space), it follows from the properties of the holomorphic functional calculus that for any
holomorphic function $g\dd D \to \CCC$ (where $D$ is an open neighbourhood of $\lambda$ in $\CCC$)
and each $\epsi > 0$ there is $\delta > 0$ such that for every bounded operator $T$ on $\ell_2$ with
$\|T - \lambda I\| \leqsl \delta$ one has $\spc(T) \subset D$ and $\|f[T] - f[\lambda I]\| \leqsl
\epsi$ (see also the proof of \PRO{holo} below). We leave it as an exercise that the assertion
of (ii) now easily follows.\par
Now assume (i) is fulfilled. We infer from \THM{gen-cont} and \PRO{B-C} that $f \in
\DDC^{\infty}(\Omega)$. We want to show (iii). To this end, fix $z \in \Omega'$ and take $\epsi_z >
0$ and $M_z > 0$ such that $\|f[X]\| \leqsl M_z$ whenever $X \in \ddD_n(\Omega)$ is such that
$\|X - z I_n\| \leqsl 2 \epsi_z$ (where $n$ is arbitrary)---see (i). We claim that for each
$n \geqsl 1$,
\begin{align}\label{eqn:D-b}
|\Delta(\lambda_1,\ldots,\lambda_n)f| \leqsl \frac{M_z}{\epsi_z^{n-1}} \qquad & \textup{if }
(\lambda_1,\ldots,\lambda_n) \in \Omega^{(n)} \textup{ and}\\& |\lambda_j - z| \leqsl \epsi_z\
(j=1,\ldots,n).\notag
\end{align}
Observe that the above inequality is immediate for $n = 1$. To prove \eqref{eqn:D-b} for $n > 1$,
denote by $A$ the $n \times n$ matrix $[a_{p,q}]$ with $a_{j+1,j} = 1$ ($j=1,\ldots,n-1$) and
$a_{p,q} = 0$ otherwise, and note that if $\lambda_1,\ldots,\lambda_n$ are as in \eqref{eqn:D-b},
then $X := \epsi_z A + \Diag(\lambda_1,\ldots,\lambda_n)$ belongs to $\ddD_n(\Omega)$, $\|X - z
I_n\| \leqsl 2 \epsi_z$ and consequently $\|f[X]\| \leqsl M_z$. So, the bottom left corner of $X$,
say $b$, satisfies the inequality $|b| \leqsl M_z$. But $b = \Delta(\lambda_1,\ldots,\lambda_n)f
\cdot \epsi_z^{n-1}$, which yields \eqref{eqn:D-b}.\par
Since $f \in \DDC^{\infty}(\Omega)$, there are continuous functions $F_n\dd \Omega^{[n]} \to \CCC$
($n=1,2,\ldots$) such that $F_n(z_1,\ldots,z_n) = \Delta(z_1,\ldots,z_n)f$ for any $(z_1,\ldots,z_n)
\in \Omega^{(n)}$ (see \PRO{DDC}). Further, it follows from \PRO{DD-T} and its proof that for any
$w \in \Omega$, $z \in \Omega'$ and $n \geqsl 1$:
$$f(w) = \sum_{k=0}^n \frac{f^{(k)}(z)}{k!}(w-z)^k + (w-z)^n [F_{n+1}(z,\ldots,z,w)
- F_{n+1}(z,\ldots,z,z)].$$
The continuity of $F_n$ and the density of $[B(z,\epsi_z)]^n \cap \Omega^{(n)}$ in $B[(z,\epsi_z)]^n
\cap \Omega^{[n]}$ combined with \eqref{eqn:D-b} yield that $|F_n(\lambda_1,\ldots,\lambda_n)|
\leqsl M_z/\epsi_z^{n-1}$ for any $(\lambda_1,\ldots,\lambda_n) \in [B(z,\epsi_z)]^n \cap
\Omega^{[n]}$ and $z \in \Omega'$. So, Taylor's expansion of $f$, stated above, may be estimated
as follows:
$$|f(w) - \sum_{k=0}^n \frac{f^{(k)}(z)}{k!}(w-z)^k| \leqsl 2M_z
\Bigl(\frac{|w - z|}{\epsi_z}\Bigr)^n$$
whenever $z \in \Omega'$, $w \in \Omega$ and $|z - w| < \epsi_z$. Consequently,
\begin{equation}\label{eqn:holo}
f(w) = \sum_{n=0}^{\infty} \frac{f^{(n)}(z)}{n!} (w-z)^n \qquad (z \in \Omega',\ w \in \Omega,\
|z - w| < \epsi_z).
\end{equation}
For any $z \in \Omega'$ fix $w_z \in \Omega$ such that $\delta_z := |z - w_z| \in (0,\frac12
\epsi_z)$. Since the series appearing in \eqref{eqn:holo} converges for $w = w_z$, we infer that
the assignment $w \mapsto \sum_{n=0}^{\infty} \frac{f^{(n)}(z)}{n!} (w - z)^n$ well defines
a holomorphic function $f_z\dd B(z,\delta_z) \to \CCC$ (which extends the restriction of $f$
to $B(z,\delta_z) \cap \Omega$). Put $U = \bigcup_{z\in\Omega'} B(z,\delta_z)$. We claim that
all the functions $f_z$ ($z \in \Omega'$) agree. Indeed, if the domains of $f_{z_1}$ and $f_{z_2}$
(for some $z_1, z_2 \in \Omega'$), that is---the balls $B(z_1,\delta_{z_1})$ and
$B(z_2,\delta_{z_2})$, meet, then $|z_1 - z_2| < \delta_{z_1} + \delta_{z_2}$. Without loss
of generality, we may assume $\delta_{z_2} \leqsl \delta_{z_1}$. But then $|z_1 - z_2| <
2 \delta_{z_1} < \epsi_{z_1}$ and consequently $z_2 \in B(z_1,\epsi_{z_1}) \cap \Omega' \cap
B(z_2,\epsi_{z_2})$. So, the set $A := B(z_1,\epsi_{z_1}) \cap B(z_2,\epsi_{z_2}) \cap \Omega$ has
a non-isolated point (namely, $z_2$) and $f_{z_1}\bigr|_A = f\bigr|_A = f_{z_2}\bigr|_A$, which
implies that $f_{z_1}$ and $f_{z_2}$ coincide on the whole intersection of their domains
(by the identity principle).\par
The property established above enables us to define a holomorphic function $F\dd U \to \CCC$
by the rule: $F(w) = f_z(w)$ for $z \in \Omega'$ and $w \in B(z,\delta_z)$. Notice that $F$ extends
$f\bigr|_{\Omega \cap U}$. Hence, if $\Omega \subset U$, the proof is complete. Now assume $\Omega
\not\subset U$ and write $U \setminus \Omega = \{\lambda_n\dd\ 1 \leqsl n < N\}$ where $N \in \{1,2,
\ldots,\infty\}$ and $\lambda_n$'s are different (recall that $U \setminus \Omega$ is finite
or countable since $\Omega' \subset U$). Since $\lambda_n \notin \Omega'$, there is a real constant
$\rho_n \in (0,\frac1n)$ such that $\Omega \cap B(\lambda_n,\rho_n) = \{\lambda_n\}$ ($n < N$).
Notice that the sets
$$U_0 := U \setminus \overline{\bigcup_{n<N} B(\lambda_n,\frac12\rho_n)},\
B(\lambda_1,\frac12\rho_1),\ B(\lambda_2,\frac12 \rho_2),\ \ldots$$
are open and pairwise disjoint and hence we may properly define a holomorphic function $g\dd D \to
\CCC$ on their union $D$ by the rule: $g(w) = F(w)$ for $w \in U_0$ and $g(w) = f(\lambda_n)$ for
$w \in B(\lambda_n,\frac12\rho_n)$ ($n < N$). It is clear that $g$ extends
$f\bigr|_{\Omega \cap D}$. So, to finish the proof, it suffices to check that $\Omega \subset D$.
Suppose, on the contrary, that there is $z \in \Omega$ which does not belong to $D$. Then $z \neq
\lambda_n$ ($n < N$) and thus $z \in U$. We conclude from the fact that $z \in U \setminus D$ that
$z$ belongs to the closure of $\bigcup_{n<N} B(\lambda_n,\frac12\rho_n)$. So, there are sequences
$(n_k)_{k=1}^{\infty}$ and $(z_k)_{k=1}^{\infty}$ of natural and complex numbers (respectively)
such that $z_k \in B(\lambda_{n_k},\frac12\rho_{n_k})$ and $\lim_{k\to\infty} z_k = z$. Passing
to a subsequence, we may assume that either $n_k = m$ for all $k$ or $\lim_{k\to\infty} n_k =
\infty$. In the first case we obtain that $|z - \lambda_m| \leqsl \frac12 \rho_m$ and thus $z \in
B(\lambda_m,\rho_m) \cap \Omega$, which is impossible (since $z \neq \lambda_m$). In the second case
we infer that $\lim_{k\to\infty} \lambda_{n_k} = z\ (\in U)$ because $|z_{n_k} - \lambda_{n_k}| <
\rho_{n_k} < \frac{1}{n_k}$. So, there is $l$ such that $\lambda_{n_l} \in U$, which is also
impossible, and we are done.
\end{proof}

\SECT{Functional calculus in infinite dimension}

Now we would like to introduce and study extended functional calculus for bounded operators
in infinite-dimensional (complex) Hilbert spaces. To this end, let us fix the notation. Whenever
$\HHh$ is a Hilbert space, by $\BBb(\HHh)$ we denote the unital ($C^*$-)algebra of all bounded
linear operators acting on $\HHh$. The spectrum of $T \in \BBb(\HHh)$ is the set $\spc(T)$ of all
scalars $\lambda \in \CCC$ such that the operator $T - \lambda I$ is noninvertible in $\BBb(\HHh)$
(here $I$ denotes the identity operator on $\HHh$). Two bounded operators $S$ and $T$ are said to be
\textit{similar} if there is an invertible operator $G \in \BBb(\HHh)$ such that $T = G S G^{-1}$.

\begin{dfn}{diag-scal}
Let $\HHh$ be a Hilbert space with a fixed orthonormal basis $\bbB = \{e_j\}_{j \in J}$. An operator
$T \in \BBb(\HHh)$ is said to be \textit{diagonal} (\textit{with respect to $\bbB$}) iff $T e_j \in
\CCC \cdot e_j$ for any $j \in J$. $T$ is called \textit{diagonalizable} iff it is similar
to diagonal. Finally, $T$ is \textit{scalar} (in the sense of Dunford and Schwartz \cite{d-s}) if it
is similar to a normal operator.\par
The sets of all (bounded) diagonalizable and scalar operators on $\HHh$ are denoted by,
respectively, $\ddD_{\HHh}$ and $\ssS_{\HHh}$.
\end{dfn}

The original definition of a scalar operator due to Dunford and Schwartz \cite{d-s} differs from
ours. However, they are equivalent (which was also established by Dunford and Schwartz, see
\cite[Theorem~XV.6.4]{d-s}). It is easily seen that the definition of a diagonal operator does
depend on the choice of an orthonormal basis, while the notion of a diagonalizable operator does
not. That is, if $T$ is similar to a diagonal operator with respect to an orthonormal basis $\bbB$,
then $T$ is similar to a diagonal operator with respect to any other orthonormal basis as well.
Finally, since diagonal operators are normal, we see that diagonalizable operators are scalar. So,
$\ddD_{\HHh} \subset \ssS_{\HHh}$. Our aim is to define extended functional calculus for scalar
operators. To this end, we recall that for every normal operator $N \in \BBb(\HHh)$ there exists
a unique spectral measure $E\dd \Bb(\spc(N)) \to \BBb(\HHh)$ (where $\Bb(A)$ is the $\sigma$-algebra
of all Borel subsets of a compact set $A \subset \CCC$) such that $N = \int_{\spc(N)} z E(\dint{z})$
(this fact is known as the \textit{spectral theorem}). Then, for every bounded Borel function $f\dd
\spc(N) \to \CCC$ one defines $f[N]$ as the integral $\int_{\spc(N)} f(z) E(\dint{z})$. With the aid
of the Fuglede-Putnam theorem \cite{fug,put} (see also \cite[Theorem~IX.6.7]{con} for a simpler
proof) one easily proves the following result on intertwining between normal operators.

\begin{lem}{F-P}
Let $N, M \in \BBb(\HHh)$ be normal and $P, G \in \BBb(\HHh)$ be invertible operators. If $P N
P^{-1} = G M G^{-1}$, then $\spc(N) = \spc(M)$ and
$$P f[N] P^{-1} = G f[M] G^{-1}$$
for any bounded Borel function $f\dd \spc(N) \to \CCC$.
\end{lem}

\LEM{F-P} enables us to define extended functional calculus for scalar operators. Let us introduce

\begin{dfn}{ext-s}
For any set $\Omega \subset \CCC$ and a Hilbert space $\HHh$ denote by $\ddD_{\HHh}(\Omega)$ and
$\ssS_{\HHh}(\Omega)$ the sets of all diagonalizable and, respectively, scalar operators on $\HHh$
whose spectra are contained in $\Omega$.\par
Let $f\dd \Omega \to \CCC$ be a Borel function which is bounded on compact subsets of $\Omega$. For
any operator $T \in \ssS_{\HHh}(\Omega)$ we define $f[T]$ as follows: take an invertible operator
$G \in \BBb(\HHh)$ such that $G T G^{-1}$ is normal and put
$$f[T] = G^{-1} f[G T G^{-1}] G.$$
\LEM{F-P} asserts that $f[T]$ is well defined, i.e. it is independent of the choice of $G$ for which
$G T G^{-1}$ is normal.
\end{dfn}

When $g(z) = \bar{z}$, the transform $T \mapsto g[T]$ (where $T$ runs over scalar operators) was
studied by us earlier \cite{pn0}, where we used other notation. For any scalar operator $T$, $g[T]$
was denoted by $T^{(*)}$ and called the \textit{quasi-adjoint} of $T$. (The quasi-adjoint was
involved there to characterize operator algebras similar to commutative $C^*$-algebras.)\par
Our last aim of the paper is to characterize those functions $f$ for which the transform $f_{\op}\dd
\ddD_{\HHh}(\Omega) \ni T \mapsto f[T] \in \BBb(\HHh)$ is continuous on some (or any)
infinite-dimensional Hilbert space $\HHh$. This is included in the next result, where we denote
by $\ell_2$ the classical separable Hilbert space. As usual, $I$ stands for the identity operator
on a suitable Hilbert space.

\begin{pro}{holo}
For a continuous function $f\dd \Omega \to \CCC$ \tfcae
\begin{enumerate}[\upshape(i)]
\item for each $\lambda \in \Omega'$ there exist positive real constants $\epsi = \epsi(\lambda)$
   and $M = M(\lambda)$ such that $\|f[K+\lambda I]\| \leqsl M$ whenever $K \in \BBb(\ell_2)$ is
   a finite-dimensional diagonalizable operator such that $\|K\| \leqsl \epsi$ and
   $\spc(K + \lambda I) \subset \Omega$;
\item $f_{\op}\dd \ddD_{\ell_2}(\Omega) \to \BBb(\ell_2)$ is continuous;
\item for an arbitrary Hilbert space $\HHh$, the function $f_{\op}\dd \ssS_{\HHh}(\Omega) \to
   \BBb(\HHh)$ extends to a continuous function of the set $\{T \in \BBb(\HHh)\dd\ \spc(T) \subset
   D\}$ for some open \textup{(}in $\CCC$\textup{)} set $D$;
\item $f$ extends to a holomorphic function of an open set $D \supset \Omega$.
\end{enumerate}
\end{pro}
\begin{proof}
First assume $f$ extends to a holomorphic function $F\dd D \to \CCC$. It follows from the Riesz
functional calculus that $F$ induces a holomorphic function $F_{\op}\dd U \to \BBb(\HHh)$
on the open (in $\BBb(\HHh)$) set $U := \{T \in \BBb(\HHh)\dd\ \spc(T) \subset D\}$ (consult e.g.
\cite[VII.\S4]{con} or \cite[I.\S7]{b-d}). It is easy to check that $F_{\op}$ extends $f_{\op}\dd
\ssS_{\HHh}(\Omega) \to \BBb(\HHh)$. This shows that (iii) follows from (iv). Since (i) obviously
follows from (ii), and (ii) from (iii), we only need to prove that (iv) is implied by (i). Taking
into account \THM{holo}, it suffices to verify that condition (i) of that result is satisfied. But
that condition easily follows from the condition (i) of the proposition. The details are left
to the reader.
\end{proof}

The above result has a remarkable consequence: the Riesz functional calculus (that is,
the holomorphic functional calculus) is as wide as possible when we require its continuity.

\end{document}